\documentclass{amsart}
\usepackage{etoolbox}
\ifdef{\warningsaserrors}{%
  }{}

\usepackage{amsmath}
\usepackage{amsthm}
\usepackage{amssymb}
\usepackage{amsfonts}
\usepackage{graphicx}
\usepackage[usenames,dvipsnames,svgnames,table]{xcolor}
\usepackage{esint}
\usepackage{enumerate}
\usepackage{fullpage}
\usepackage{verbatim}
\usepackage{booktabs}

\allowdisplaybreaks

\theoremstyle{plain}
\newtheorem{thm}{Theorem}
\newtheorem*{conj}{Conjecture}
\newtheorem{cor}[thm]{Corollary}

\newtheorem{lem}[thm]{Lemma}
\newtheorem{prop}[thm]{Proposition}

\theoremstyle{remark}
\newtheorem{rmk}{Remark}

\newcounter{QuestionCounter}

\newcommand{\R}{\mathbb{R}}
\newcommand{\C}{\mathbb{C}}
\newcommand{\N}{\mathbb{N}}
\newcommand{\Z}{\mathbb{Z}}
\renewcommand{\S}{\mathbb{S}}

\newcommand{\SK}{\mathcal{K}}
\newcommand{\SQ}{\mathcal{Q}}
\newcommand{\SI}{\mathcal{I}}
\newcommand{\SO}{\mathcal{O}}
\newcommand{\SR}{\mathcal{R}}
\newcommand{\SF}{\mathcal{F}}

\newcommand{\bigO}{\mathcal{O}}

\newcommand{\rgamma}{{\hat\gamma}}
\newcommand{\rnu}{\hat\nu}
\newcommand{\rtau}{\hat\tau}
\newcommand{\rs}{\hat s}
\newcommand{\rk}{\hat k}
\newcommand{\ork}{\overline{\hat k}}

\newcommand{\rG}{\hat G}

\newcommand{\rL}{\hat L}
\newcommand{\rA}{\hat A}
\newcommand{\rSQ}{\hat\SQ}
\newcommand{\rSI}{\hat\SI}

\newcommand{\ol}[1]{\overline{#1}}
\newcommand{\p}{\partial}
\renewcommand{\t}[1]{\widetilde{#1}}

\newcommand{\dd}[1]{\;d#1}

\renewcommand{\k}{\overline{k}}
\newcommand{\kv}{\overline{k}}
\newcommand{\IP}[2]{\langle #1, #2\rangle}
\newcommand{\vn}[1]{\lVert #1 \rVert}

\begin{document}

\title{Theory and Numerics for Chen's flow of curves}
\author{M.K.\ Cooper, G.\ Wheeler and V.-M.\ Wheeler}
\date{\today}

\begin{abstract}
In this article we study Chen's flow of curves from theoreical and numerical
perspectives.
We investigate two settings: that of closed immersed $\omega$-circles, and
immersed lines satisfying a cocompactness condition.
In each of the settings our goal is to find geometric conditions that allow us
to understand the global behaviour of the flow:
for the cocompact case, the condition is straightforward and the argument is
largely standard.
For the closed case however, the argument is quite complex. The flow shrinks
every initial curve to a point if it does not become singular beforehand, and
we must identify a condition to ensure this behaviour as well as
identify the point in order to
perform the requisite rescaling.
We are able to successfully conduct a full analysis of the rescaling under a
curvature condition. The analysis resembles the case of the mean curvature flow
more than other fourth-order curvature flow such as the elastic flow or the
curve diffusion flow, despite the lack of maximum and comparison principles.
Our work is informed by a numerical study of the flow, and we include a section
that explains the algorithms used and gives some further simulations.
\end{abstract}

\maketitle

\section{Introduction}

In \cite{BWWChen}, Chen's flow
\begin{equation}
\label{CF}
\tag{CF}
(\partial_tf)(p,t) = -(\Delta^2 f)(p,t)\,,\qquad (p,t) \in M^n\times(0,T)
\,,
\end{equation}
was proposed and studied.
Here $f:M^n\times[0,T)\rightarrow\R^{N}$ (with $f(p,0) = f_0(p)$ for a given
smooth isometric immersion $f_0:M^n\rightarrow\R^{N}$) is a one-parameter family of smooth isometric immersions,
and $\Delta^2$ is Chen's biharmonic operator, also known as the iterated rough Laplacian.
Chen's conjecture is that $\Delta^2 f \equiv 0$ implies $\Delta f \equiv 0$, a
desirable fact that arises in the study of the spectral decomposition of
immersions in submanifold theory.
There has been much activity on the conjecture (see as a sample the recent
papers \cite{b08,d06,luo14,m14,mon06,n14,ou10,ou16,ou12,w14,whe13}
and Chen's recent survey \cite{chen13}), but still it remains open.

Local well-posedness applies from \cite{BWWChen} to all dimensions, and so the
key issue to investigate is asymptotic behaviour of the flow.
The results of \cite{BWWChen} on this are primarily for two and four
dimensional evolving submanifolds.
The convergence result \cite[Theorem 4]{BWWChen} is for surfaces only and
relies on an analysis of the $L^2$-norm of the tracefree second fundamental
form that involves studying a sequence of parabolic rescalings 
about a singularity.
Only in two dimensions is the $L^2$-norm of the tracefree second fundamental
form scale-invariant, and so this approach can not work in general.

\begin{figure}[t]
\begin{center}
\includegraphics[scale=0.6]{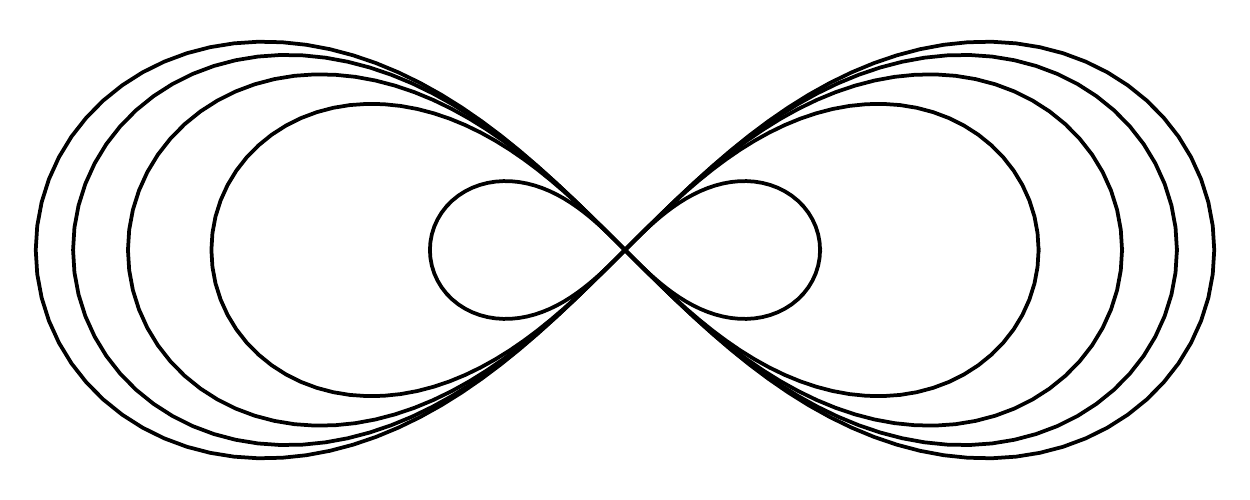}
\caption{Numerical solution of the shrinking Lemniscate of Bernoulli whose initial data satisfy $(x^2 + y^2)^2 = x^2 - y^2$.}\label{fig:ShrinkingLemniscate}
\end{center}
\end{figure}

\begin{figure}[t]
\begin{center}
\includegraphics[scale=1.3]{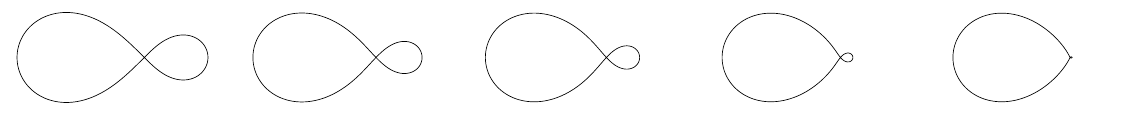}
\caption{A non-symmetric (unbalanced) Lemniscate blowing up.}\label{fig:NonSymLem}
\end{center}
\end{figure}
\begin{figure}[t]
\begin{center}
\includegraphics[scale=2]{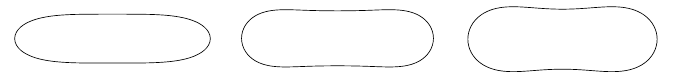}
\caption{
Evolution of the Lam\'e curve $\frac{x^4}{256} + y^2 = 1$ under Chen's flow.
The initial data is strictly convex, but the solutions becomes non-convex under the flow.
}\label{fig:Lame}
\end{center}
\end{figure}

\begin{figure}[t]
\begin{center}
\includegraphics[scale=2.3]{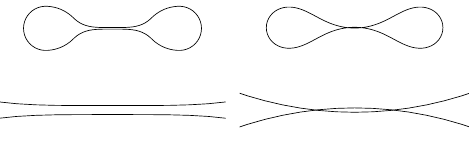}
\caption{
Evolution of a dumbbell with a narrow neck.
This shows that loss of embeddedness may occur under Chen's flow.
The bottom row is a closeup of the neck.
The top component of the neck is the graph of $x \mapsto x^4/4 + c/2$, where $c > 0$, and the bottom component is the reflection of this through the $x$-axis.
The width of the neck in the middle is $c$, which in this example is $0.04$.
The main observation that makes this example work is that the graph of $x \mapsto x^4/4 + c/2$ initially moves down in the middle, and hence if $c$ is sufficiently small it will cross the $x$-axis and intersect its reflection.
We connect these neck portions in a somewhat arbitrary way, since this is really a local phenomenon.
}\label{fig:ThinDumbbell}
\end{center}
\end{figure}

\begin{figure}[t]
\begin{center}
\includegraphics[scale=1.3]{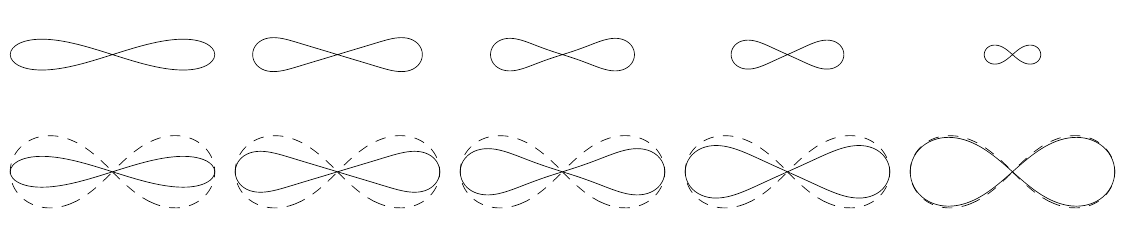}
\caption{
In the top row we have the Chen's flow evolution from a Lemniscate of Gerono $(x^2 + y^2)^2 = x^2 - 10y^2$.
The bottom row is the rescaling of the corresponding top row curve so that its horizontal extent is two, and the dashed curve is the Lemniscate of Bernoulli $(x^2 + y^2)^2 = x^2 - y^2$.
These rescalings appear to converge to the Lemniscate of Bernoulli. This indicates that
the Lemniscate of Bernoulli may be an attractor, or at least stable, for symmetric figure-eight curves under Chen's flow.
}\label{fig:Gerono}
\end{center}
\end{figure}

\begin{figure}[t]
\begin{center}
\includegraphics[scale=2]{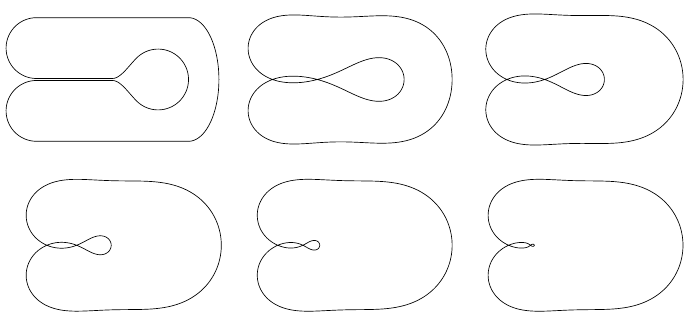}
\caption{
Chen's flow evolution of Mayer's curve.
Observe that this provides another example of the potential loss of embeddedness under Chen's flow.
}\label{fig:Mayer}
\end{center}
\end{figure}

\begin{figure}[t]
\begin{center}
\includegraphics[scale=1.25]{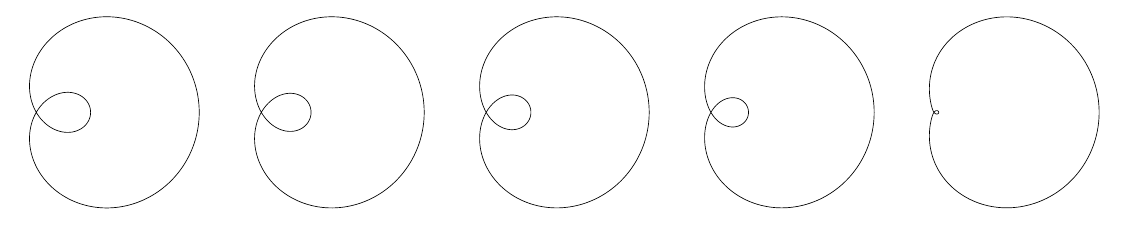}
\caption{
Chen's flow evolution of a Lima\c{c}on.
Observe that the flow quickly cuts off the smaller loop.
}\label{fig:loops}
\end{center}
\end{figure}

\begin{figure}[t]
\begin{center}
\includegraphics[scale=0.75]{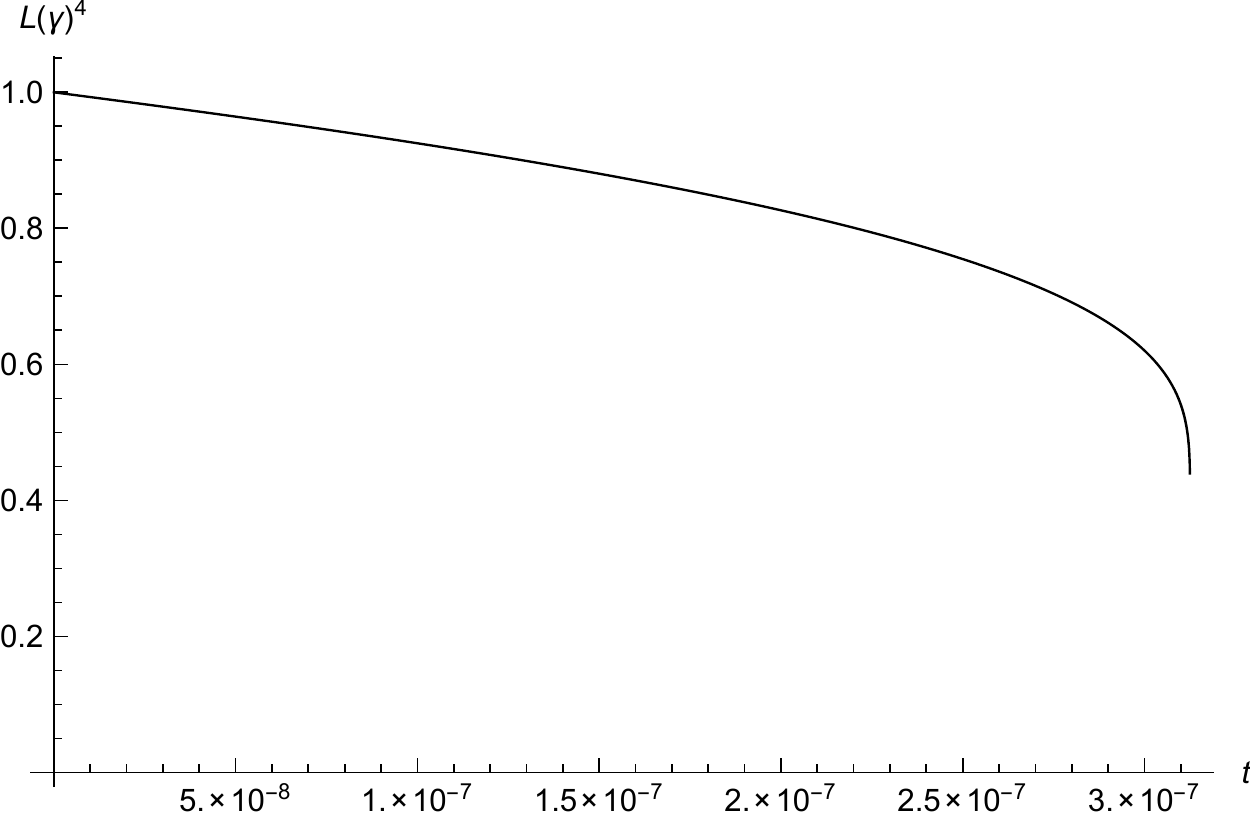}
\caption{
This displays is the evolution of $L[\gamma]^4$ over the lifetime of the Lima\c{c}on from Figure~\ref{fig:loops}.
Notice that it is not convex.
}\label{fig:LoopsLength}
\end{center}
\end{figure}

\begin{figure}[t]
\begin{center}
\includegraphics[scale=1]{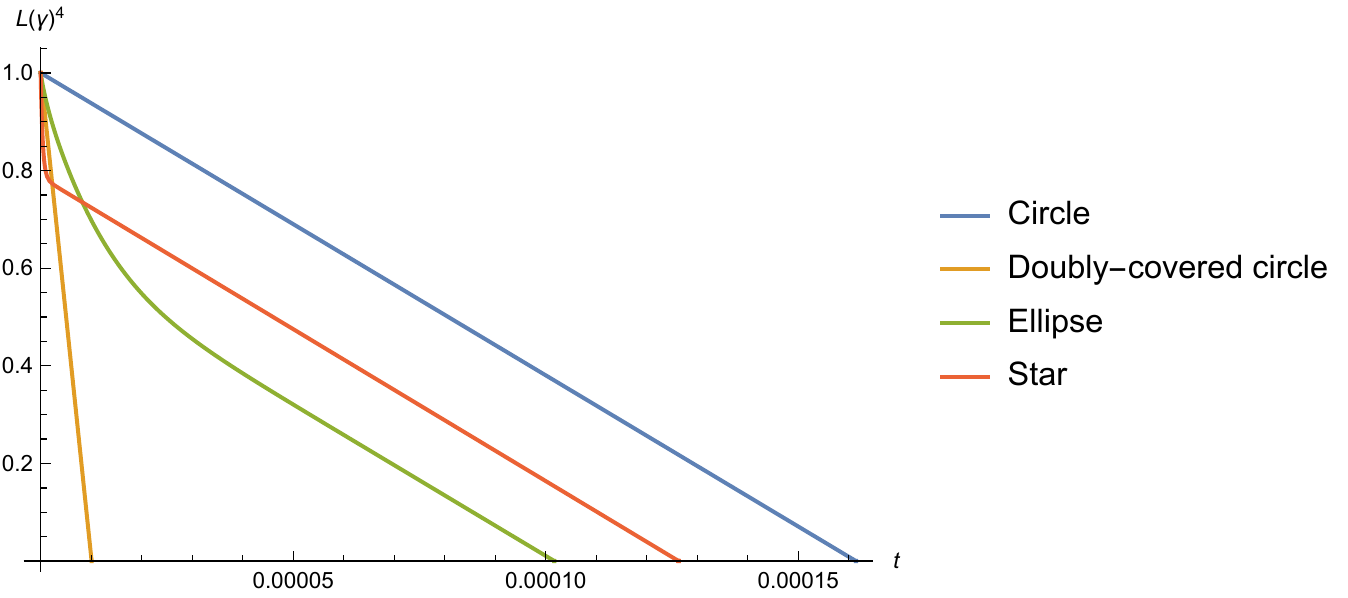}
\caption{
  This displays the evolution of $L[\gamma]^4$ over the lifetime of various curves evolving under Chen's flow.
  Note that all of these curves shrink to a round circle under Chen's flow.
  The initial data for the first two evolutions are the circle and doubly-covered circle, respectively, each scaled so that their initial length is one.
  The initial data for the third evolution is an ellipse of length one such that the ratio of the semi-major and semi-minor axes is two.
  The initial data for the fourth evolution is equivalent to, up to scaling, the star-shaped curve given in polar coordinates by $r = 0.1\cos(5\theta) + 1$.
  This star-shaped curve is scaled so that its length is one.
Note that the graph is convex in each case, and that the lifespan of the flow is bounded from above by that of the circle.
}\label{fig:CirclePertLength4}
\end{center}
\end{figure}

This article is concerned with the one-dimensional case of Chen's flow:
\begin{equation}
\label{ChenFlow}
	\partial_t\gamma = -\gamma_{s^4} = -(k_{ss} - k^3)\nu + 3kk_s\tau
	\,.
\end{equation}
Here $\gamma:I\times[0,T)\rightarrow\R^2$ is a one-parameter family of immersed
curves, where $I$ is either $\S$ (and then $\gamma$ is closed) or $I = \R$ and
$\gamma$ is an entire curve. We have used $k$ for the signed scalar curvature,
the subscript $s$ to denote differentiation with respect to arc-length $s$, and
$\{\nu,\tau\}$ are the normal and tangent vectors along $\gamma$.
Note that the tangential term could only possibly impact the shape of the
evolving family $\gamma$ in the non-closed case.
In the closed, embedded case our convention is that the unit normal points
toward the interior of $\gamma$ and that $\gamma_{ss} = k\nu$.
The vector $\partial_t\gamma$ is pointwise orientation independent, so in the
case where $\gamma$ is not embedded or not closed, we choose an arbitrary
orientation (but retain our convention that $\gamma_{ss} = k\nu$).

Chen's flow, being a fourth-order curvature flow, behaves differently to typical second-order flows.
For example, compared to the curve-shortening flow, we generically have loss of strict
convexity (see Figure~\ref{fig:Lame}) and loss of embeddedness (see
Figure~\ref{fig:ThinDumbbell} and Figure~\ref{fig:Mayer}) phenomena (see
\cite{Blatt} for a general description of these phenomena).
However there are some deep analogies between Chen's flow (which is the
biharmonic heat equation for the immersion) and the curve-shortening flow
(which is the heat equation for the immersion), as we shall explain.

Although the one-dimensional context is considerably simpler than the
higher-dimensional cases considered in \cite{BWWChen}, a new idea is required
to carry out the analysis.
In one dimension the tracefree second fundamental form is not defined.
The natural analogue is the (normalised) oscillation of curvature, either in
$L^1$ or in $L^2$:
\[
	\SK_{1,osc} = \int_\gamma |k-\k|\,ds\qquad
	\SK_{2,osc} = L[\gamma]\int_\gamma |k-\k|^2\,ds
	\,.
\]
Here $k$ is the signed scalar curvature of a planar immersed curve
$\gamma:I\rightarrow\R^2$, $ds$ is its arc-length element, $L$ denotes the
total arclength, and $\k$ is the average of the curvature.

The oscillation of curvature has been used
recently for fourth-order flows \cite{PW16,PW19,W13,WW17} that are similar in form to
Chen's flow studied here.
However, as with curve-shortening flow, the oscillation of curvature does not behave well under Chen's flow of
curves.
The isoperimetric ratio, another quantity used both
classically \cite{A98,EG87,H98,T98,G90} and for fourth-order flows \cite{C03,PW16,W13}, also does
not behave well under Chen's flow (unless we are in a very special situation,
see Proposition \ref{PNdecayrsi}).
The isoperimetric defect on the other hand is monotone under the curve-shortening flow \cite{G83}.
Under some hypotheses, the defect is also monotone under Chen's flow.
The defect is however not scale-invariant and so gives quite weak information on the blowup.
It is not yet clear to us how we
might use this monotonicity for our analysis.
This behavior is reminiscent of the curve shortening flow \cite{G83}.

Chen's flow is not a gradient flow for length in the $L^2$-sense (as
the curve shortening and mean curvature flows are), or in the $H^{-1}$-sense
(as the curve diffusion and surface diffusion flows are).
However the flow does decrease length at a dramatic rate, satisfying
\[
	L'(t) = - \int_\gamma k_s^2 + k^4 \,ds
	\,.
\]
This is a higher-order analogue of the curve shortening flow, which satisfies $L' = -\int_\gamma k^2\,ds$.
Continuing with this similarity, it is shown in \cite[Theorem 2]{BWWChen} that (here we state the one-dimensional case only)
\[
	L(t) \le \sqrt[4]{L_0^4 - 64\omega^4\pi^4 t}
	\,.
\]
Above we used $\omega = \frac1{2\pi}\int_\gamma k\,ds$ to denote the winding
number of $\gamma$ and we have assumed that $\gamma$ is closed.
Among other conclusions, this shows that the flow exists at most for finite time, with the estimate
\begin{equation}
\label{FT}
	T \le \frac{L_0^4}{64\omega^4\pi^4} = \frac14\Big(\frac{L_0}{2\omega\pi}\Big)^4\qquad\text{for $\omega\ne0$, and}\qquad
	T \le \frac{L_0^4}{64\pi^4}\qquad\text{for $\omega=0$ (and $\gamma$ closed)}
\,.
\end{equation}
This estimate is sharp for $\omega\ne0$, as we shall now demonstrate.

Consider the flow of a round $\omega$-circle with initial radius $r_0>0$, $\rho:\S\times[0,T)\rightarrow\R^2$.
By explicit calculation, we find that the radius function $r:[0,T)\rightarrow\R$ satisfies
\[
	(r^4(t))' = -4
\]
which implies
\[
	T = \frac{r_0^4}{4}
	\,.
\]
Since for this flow we have $L_0 = 2\omega\pi r_0$, or $r_0 =
\frac{L_0}{2\omega\pi}$, this achieves equality in the estimate \eqref{FT},
showing that \eqref{FT} is sharp, and furthermore is sharp in every regular homotopy class of an $\omega$-circle.
The estimate \eqref{FT} holds however also in the case where $\gamma$ is
topologically a lemniscate (note that if $\gamma$ is not closed and $\omega=0$
we may have global existence, to for instance a straight line), and there the
sharpness is not clear.

We do however have a natural candidate for the optimal bound when $\gamma$ is
topologically a lemniscate.
The \emph{Lemniscate of Bernoulli} $\beta:\S\rightarrow\R^2$ (by which we mean
the parametrisation of the solution set to $(x^2+y^2)^2 = x^2 - y^2$ where
$(x,y)\in\R^2$) satisfies the relation
\begin{equation}
\label{LB1}
	\IP{\beta}{\nu} = \frac16 k_{ss}
	\,.
\end{equation}
This was first described in \cite{EGMWW15}, where this condition is
enough to guarantee that the Lemniscate shrinks homothetically under curve
diffusion flow.
Rather surprisingly, Bernoulli's Lemniscate also satisfies the relation
\begin{equation}
\label{LB2}
	\IP{\beta}{\nu} = -\frac{1}{27} k^3
	\,.
\end{equation}
These relations allow us to show that the Lemniscate of Bernoulli is a homothetically shrinking solution to Chen's flow.
Setting $h(t) = (1-132t)^\frac14$, we find that $\gamma(\theta, t) = h(t)\beta(\theta)$
satisfies $\gamma(\theta,0) = \beta(\theta)$
 and 
\begin{align*}
	\IP{\gamma_t - (-\gamma_{s^4})}{\nu}
	&= \IP{\gamma_t + (k_{ss} - k^3)\nu}{\nu}
	= h'(t)\IP{\beta}{\nu} + h^{-3}(t) (6\IP{\beta}{\nu} + 27\IP{\beta}{\nu})
	\\
	&= \frac{h^{-3}(t)}{4}\IP{\beta}{\nu}\Big[ -132 + (24 + 108)\Big] = 0
\end{align*}
making $(\theta,t)\mapsto h(t) \beta(\theta) = \gamma(\theta,t)$ a homothetic
solution to Chen's flow.\footnote{For details of this calculation we refer the
interested reader to Appendix B. We also give a conjecture related to the
Lemniscate of Bernoulli and Chen's flow there.}
We have numerically simulated this solution in Figure \ref{fig:ShrinkingLemniscate}.
We are not aware of any deep, geometrical reason why the Lemniscate of Bernoulli satisfies the relations \eqref{LB1}, \eqref{LB2}.
Other lemniscates, for instance the Lemniscate of Gerono, does not satisfy either of these relations.
Numerically, the evolution of Gerono's Lemniscate is asymptotic to that of Bernoulli's -- see Figure \ref{fig:Gerono}.

We therefore find examples of homothetic solutions for any winding number
$\omega$. All are circular except for when the winding number is zero,
where it appears that the canonical example is the Lemniscate of Bernoulli.
These are the mildest form of finite-time singularity, where the flow exists until the length has disappeared.
There are many other kinds of finite-time singularities, where the flow cuts
off loops and cusps (see Figure \ref{fig:loops}), or such cusps form where there were none previously, see
for instance the evolution of Mayer's curve (see Figure \ref{fig:Mayer} for our
simulation and \cite{Mayer2001} for the original article, there used for the
area preserving curve shortening flow).

The Lemniscate does not appear stable without additional considerations, for
instance, symmetry requriements (see Figure \ref{fig:NonSymLem}).
Numerically multiply-covered circles appear to be stable under small, symmetric perturbation for Chen's flow.
Our aim is not to prove a perturbative result however, but instead to find a
curvature condition that guarantees the flow is asymptotic to a shrinking
$\omega$-circle, in analogy with the surface case in \cite{BWWChen}.
Curvature conditions are quite rough measurements; such an aim is doomed in the
case of the Lemniscate, because a balancing symmetry condition can not be
detected by norm restrictions on the curvature (see Figure \ref{fig:NonSymLem} for the evolution of an unbalanced Lemniscate of Bernoulli).
Similarly, a cardioid (which can be the result of an arbitrarily small but not
symmetric deformatiuon of an $\omega$-circle) will flow quickly to cut off the
smaller loop (Figure \ref{fig:loops}).
Even in the class of embedded initial data, the development of a finite-time non-circular curvature singularity can occur.
This can be seen by studying Chen's flow of Mayer's aforementioned example (see Figure \ref{fig:Mayer}).
This suggests that curvature conditions that involve more curvature than $(\omega+1)$ round circles are likely false.

The key idea to control the evolving geometry of the flow is the following.
Chen's flow decreases length at a dramatic rate.
However length is certainly not scale-invariant.
In order to use this fact, we investigate the acceleration of the fourth power
of length, and show that this quantity (under an initial condition) has a
guaranteed sign.
In effect, this amounts to a fundamental observation about Chen's flow: that
the fourth power of length is, under some conditions, a convex function of
time.
Then, since the velocity of the fourth power of length also happens to be
scale-invariant, we may use this observation to obtain key estimates for a
blowup of the solution around the finite-time singularity.
This eventually results in the following theorem, which is the main
result of the paper.

Let $\gamma:I\times[0,T)\rightarrow\R^2$ be a smooth one-parameter family of immersed curves.
For the statement we make the definition
\[
	\SQ:[0,T)\rightarrow\R\,,\qquad \SQ(t) := L^3\int_\gamma k_s^2 + k^4\,ds\,.
\]

\begin{thm}
\label{TMmain}
Suppose $\gamma:I\times[0,T)\rightarrow\R^2$ is a compact Chen flow with
initial data $\gamma_0:I\rightarrow\R^2$ being an immersed $\omega$-circle
satisfying
\begin{equation*}
	\SQ(0) = L^3\int_\gamma k_s^2 + k^4\,ds\bigg|_{t=0} < 16\omega^4\pi^4 + \varepsilon_1
\end{equation*}
where $\varepsilon_1$ is the universal constant depending only on $\omega$ from
\eqref{EQseteps1}.
Then
$L(t) \searrow 0$ as $t\nearrow T$, and $\gamma(I,t)\rightarrow\{\SO\}$, where
$\SO\in\R^2$, in the Hausdorff metric on subsets of $\R^2$.

Convergence of the flow $\gamma$ to the point $\SO$ is asymptotically smooth
and the point is round in the following sense.
Let $\rgamma:I\times[0,\infty)\rightarrow\R^2$ be the rescaling of $\gamma$
around the final point $\SO$ given by
\[
\hat\gamma(u, t) = (4T)^{-\frac14}e^{t}\Big(\gamma(u,T - Te^{-4t}) - \SO\Big)
\,.
\]
The rescaled flow $\rgamma(I,t)$ converges to the unit $\omega$-circle centred at the
origin, with curvature $\hat k$ converging exponentially fast toward its average and
all derivatives of curvature covnerging exponentially fast to zero.
\end{thm}

We also consider the cocompact case.
This is simpler than the compact case, since Chen's flow tends to straighten
(which is a smoothing effect) initial data that is already sufficiently flat,
instead of driving every initial curve to a finite-time singularity.
It is possible to obtain a curvature condition in $L^2$ that guarantees convergence to a flat line.
There is no rescaling required.

For the statement we make the definition
\[
	\SR:[0,T)\rightarrow\R\,,\qquad \SR(t) := L\int_\gamma k^2\,ds\,.
\]
\begin{thm}
\label{TMcoco2}
Let $\gamma:I\times[0,T)\rightarrow\R^2$ be a cocompact Chen flow with initial
data $\gamma_0:I\rightarrow\R^2$ being an immersed interval satisfying $\omega = 0$ and
\begin{equation*}
\SR(0) = L\int_{\gamma} k^2\,ds\bigg|_{t=0} \le \frac34\,.
\end{equation*}

There exists a straight line $\gamma_\infty:\R\times[0,T)\rightarrow\R^2$
such that $\gamma\rightarrow \gamma_\infty$ exponentially fast in the $C^\infty$ topology.
\end{thm}

This article is organised as follows.
In Section \ref{sect:Evolution} we give the evolution equations for essential quantities under the flow.
Then, in Section \ref{sect:Estimates} we prove estimates for key quantities
along the flow, including that the fourth power of length is convex, under the
curvature condition of Theorem \ref{TMmain}.
In this section we also characterise the maximal existence time.
We show that the $L^2$-norm of the curvature blows up at a specific rate and
that $L(t)\rightarrow0$ as $t\nearrow T$, allowing us to identify a `final
point' (Theorem \ref{TMblowup} and Theorem \ref{TMfinalpoint}).
This is similar to the classical case of mean curvature flow \cite{Huisken} and
curve-shortening flow \cite{GH86}.

In order to illustrate the structure of the more complex argument to come for
the compact case, we next deal with in Section \ref{sect:cocompact} the
cocompact case, and prove Theorem \ref{TMcoco2}.
This can be contrasted with similar strategies in \cite{MWW17,WW17}.

Section \ref{sect:Asymptotics} considers a continuous rescaling at the final
time (for the compact case).
In this sections we switch between the original and rescaled flow as required,
and present a number of arguments that eventually yield the required
convergence.
This involves uniform estimates on length, the rescaled position vector, as
well as curvature and its derivatives.
A decisive role is played by two key estimates in this part: First, the control
on $|(L^4)'(t)|$, which is scale-invariant and persists into the rescaling.
Second, the isoperimetric ratio makes an appearance, which (combined with our
other control) allows us to obtain the requisite decay to finish the
convergence argument.

Finally, Section \ref{sect:Numerics} describes our numerical scheme and some
further simulations we conducted of the flow.

\section*{Acknowledgements}

The third athor acknowledges support from ARC Discovery Project DP180100431 and ARC DECRA DE190100379.
All of the authors are indebted to Laureate Professor Ben Andrews at the ANU,
with whom the topic of the paper was discussed at length.


\section{Evolution equations and Local existence}
\label{sect:Evolution}

First we state a straightforward existence theorem that is enough for our purposes here.

\begin{thm}
Suppose $I = \R$ or $I = \S$.
Let $\gamma_0:I\rightarrow\R^2$ be a smooth, planar immersed curve.
In the case where $I = \R$, assume the cocompactness condition
\begin{equation}
\label{CoComp}
\gamma(m+u) = me_1 + \gamma(u)\qquad\text{ where $m\in\Z$ is any integer and $e_1 = (0,1)$}
\,.
\end{equation}

There exists a smooth one-parameter family $\gamma:I\times[0,T)\rightarrow\R^2$ of immersed curves solving Chen's flow
\[
	\partial_t\gamma = -\partial_s^4\gamma
\]
with $\gamma(u,0) = \gamma_0(u)$ that are in the case of $I = \R$ cocompact according to \eqref{CoComp} and in the case of $I = \S$ closed.

If $T<\infty$ then $\vn{\gamma(t)}_{C^\infty(I)} \rightarrow \infty$ as $t\rightarrow T$.
\end{thm}

In the above we use $C^\infty(I)$ to denote either the standard $C^\infty$ space if $I = \S$, or, in the case where $I=\R$, $C^\infty(I) = C^\infty((0,1))$.

The proof of this result follows by writing the flow as a graph over the initial curve, then considering the flow as the evolution of the distance function from the initial curve.
For a short time this can be solved by standard theory, and then iterating this construction gives the characterisation of the maximal existence time above.

We now give the essential evolution equations. The proof is standard, and so we omit it.

\begin{lem}\label{LMevo}
Suppose $\gamma:I\times[0,T)\rightarrow\R^2$ evolves by Chen's flow. Then
\begin{align*}
\partial_t\, ds
&= (Fk + G_s)\, ds
\\
[\partial_t,\partial_s]
&= -(Fk+G_s)\partial_s
\\
\partial_tk
&= -F_{ss}-Fk^2 + Gk_s
\\
\partial_tk_s
&= - F_{s^3} - F_sk^2 - 3Fkk_s + Gk_{ss}
\end{align*}
where $\partial_t\gamma = -F\nu + G\tau$, that is $F = k_{ss} - k^3$ and $G = 3kk_s$.
\end{lem}

The evolution of length is as follows.

\begin{lem}\label{LMlengthevo}
Suppose $\gamma:I\times[0,T)\rightarrow\R^2$ evolves by Chen's flow. Then
\[
	(L^4)'(t) = -4L^3\int_\gamma k_s^2 + k^4\,ds = -4\SQ(t)
	\,,
\]
and
\begin{equation}\label{EQl4evo}
\begin{split}
	-(L^4)''(t) &= 4\SQ'(t) =
	- 8L^3 \int_\gamma k_{s^3}^2 \,ds
	- 80L^3 \int_\gamma k_{ss}^2k^2 \,ds
	+ 60L^3 \int_\gamma k_s^4 \,ds
	- 88L^3 \int_\gamma k_s^2k^4 \,ds
\\&\qquad
	-12L^2\vn{k_s}_2^4
	-24L^2\vn{k_s}_2^2\vn{k}_4^4
	+ 12L^2\Big( L\vn{k}_8^8 - \vn{k}_4^8 \Big)
\,.
\end{split}
\end{equation}
\end{lem}
\begin{proof}
From the evolution equations (Lemma \ref{LMevo}) we find
\begin{equation}
\label{EQlengthevo}
\begin{split}
L'
&= \frac{d}{dt} \int_\gamma ds
\\
&= \int_\gamma Fk + G_s ds
\,.
\end{split}
\end{equation}
In both the cocoompact and compact cases, integration by parts implies that
\[
	L' = \int_\gamma kk_{ss} - k^4\,ds = -\int_\gamma k_s^2 + k^4\,ds
\]
as claimed.

To prepare for the second part we derive two further evolution equations:
\begin{equation}
\label{EQevoks2}
\partial_t (k_s^2\,ds)
=
 \bigg(
 - 2F_{s^3}k_s
 - 2F_sk^2k_s
 - 5Fkk_s^2
 + \partial_s(k_s^2G)
 \bigg)\, ds
\end{equation}
and
\begin{equation}
\label{EQevok4}
\partial_t (k^4\,ds)
= \bigg(
   	- 4F_{ss}k^3 - 3Fk^5
	+ \partial_s(k^4G)
	\bigg)\, ds
\,.
\end{equation}
We start with
\begin{align*}
\partial_t (k_s^2\,ds)
&=
 \bigg(
 2k_s
   \Big(- F_{s^3} - F_sk^2 - 3Fkk_s + Gk_{ss}\Big)
 + k_s^2 (Fk + G_s)
 \bigg)\, ds
\\
&=
 \bigg(
 - 2F_{s^3}k_s - 2F_sk^2k_s - 6Fkk_s^2 + 2k_{ss}k_sG
 + k_s^2Fk
 + k_s^2G_s
 \bigg)\, ds
\\
&=
 \bigg(
 - 2F_{s^3}k_s
 - 2F_sk^2k_s
 - 5Fkk_s^2
 + \partial_s(k_s^2G)
 \bigg)\, ds
\,,
\end{align*}
which shows \eqref{EQevoks2}.
For \eqref{EQevok4} we similarly find
\begin{align*}
\partial_t (k^4\,ds)
&= \bigg(
   	  4k^3(-F_{ss}-Fk^2 + Gk_s)
	+ k^4 (Fk + G_s)
	\bigg)\, ds
\\
&= \bigg(
   	- 4F_{ss}k^3 - 3Fk^5
	+ \partial_s(k^4G)
	\bigg)\, ds
\,.
\end{align*}
Putting together \eqref{EQevoks2} and \eqref{EQevok4}, we calculate
\begin{align*}
-(L^4)''(t)
&=
	4\frac{d}{dt}\bigg(
		L^3 \int_\gamma k_s^2 + k^4\,ds
	\bigg)
\\
&=
	4L^3\frac{d}{dt} \int_\gamma k_s^2 + k^4\,ds
	-12L^2\bigg(
		\int_\gamma k_s^2 + k^4\,ds
	\bigg)^2
\\
&=
	4L^3 \int_\gamma
 \bigg(
 - 2F_{s^3}k_s
 - 2F_sk^2k_s
 - 5Fkk_s^2
 + \partial_s(k_s^2G)
 \bigg)\, ds
\\&\qquad
	+ 4L^3 \int_\gamma
  \bigg(
   	- 4F_{ss}k^3 - 3Fk^5
	+ \partial_s(k^4G)
	\bigg)\, ds
\\&\qquad
	-12L^2\vn{k_s}_2^4
	-24L^2\vn{k_s}_2^2\vn{k}_4^4
	-12L^2\vn{k}_4^8
\\
&=
	4L^3 \int_\gamma
 \bigg(
 - 2F_sk_{s^3}
 - 2F(k^2k_{ss} + 2k_s^2k)
 - 5Fkk_s^2
 \bigg)\, ds
\\&\qquad
	+ 4L^3 \int_\gamma
  \bigg(
   	- 4F(3k^2k_{ss} + 6k_s^2k) - 3Fk^5
	\bigg)\, ds
\\&\qquad
	-12L^2\vn{k_s}_2^4
	-24L^2\vn{k_s}_2^2\vn{k}_4^4
	-12L^2\vn{k}_4^8
\\
&=
	4L^3 \int_\gamma
 \bigg(
 - 2F_sk_{s^3}
 - 14Fk^2k_{ss}
 - 33Fk_s^2k
 - 3Fk^5
	\bigg)\, ds
\\&\qquad
	-12L^2\vn{k_s}_2^4
	-24L^2\vn{k_s}_2^2\vn{k}_4^4
	-12L^2\vn{k}_4^8
\\
&=
	- 8L^3 \int_\gamma
		k_{s^3}^2
		\,ds
	- 24L^3 \int_\gamma
		k_{ss}^2k^2
		\,ds
	- 48L^3 \int_\gamma
		k_{ss}k_s^2k
		\,ds
\\&\qquad
	+ 4L^3 \int_\gamma
 \bigg(
 - 14k^2k_{ss}^2
 + 14k^5k_{ss}
 - 33k_{ss}k_s^2k
 + 33k_s^2k^4
 - 3k_{ss}k^5
 + 3k^8
	\bigg)\, ds
\\&\qquad
	-12L^2\vn{k_s}_2^4
	-24L^2\vn{k_s}_2^2\vn{k}_4^4
	-12L^2\vn{k}_4^8
\\
&=
	- 8L^3 \int_\gamma
		k_{s^3}^2
		\,ds
	- 24L^3 \int_\gamma
		k_{ss}^2k^2
		\,ds
	- 48L^3 \int_\gamma
		k_{ss}k_s^2k
		\,ds
\\&\qquad
	+ 4L^3 \int_\gamma
 \bigg(
 - 14k^2k_{ss}^2
 + 14k^5k_{ss}
 - 33k_{ss}k_s^2k
 + 33k_s^2k^4
 - 3k_{ss}k^5
	\bigg)\, ds
\\&\qquad
	-12L^2\vn{k_s}_2^4
	-24L^2\vn{k_s}_2^2\vn{k}_4^4
	+ 12L^2\Big(
		L\vn{k}_8^8 - \vn{k}_4^8
	\Big)
\\
&=
	- 8L^3 \int_\gamma
		k_{s^3}^2
		\,ds
	- 80L^3 \int_\gamma
		k_{ss}^2k^2
		\,ds
	+ 60L^3 \int_\gamma
		k_s^4
		\,ds
	- 88L^3 \int_\gamma
		k_s^2k^4
		\,ds
\\&\qquad
	-12L^2\vn{k_s}_2^4
	-24L^2\vn{k_s}_2^2\vn{k}_4^4
	+ 12L^2\Big(
		L\vn{k}_8^8 - \vn{k}_4^8
	\Big)
\,,
\end{align*}
as required.
\end{proof}

The $L^2$-norm of the curvature evolves as follows.

\begin{lem}\label{LMl2kevo}
Suppose $\gamma:I\times[0,T)\rightarrow\R^2$ evolves by Chen's flow. Then
\[
	\frac{d}{dt}\int_\gamma k^2\,ds
	 =
        - 2\int_\gamma k_{ss}^2\,ds
	- 3\int_\gamma k_s^2k^2\,ds
	+ \int_\gamma k^6\, ds
\,.
\]
\end{lem}
\begin{proof}
Using Lemma \ref{LMevo} we find
\begin{align*}
	(k^2\,ds)_t
	&=
          2k(-F_{ss}-Fk^2 + Gk_s) + k^2(Fk + G_s)\, ds
\\
	&=
          -2kF_{ss} - k^3F + (Gk^2)_s \, ds
\,.
\end{align*}
Then
\begin{align*}
\frac{d}{dt}\int_\gamma k^2\,ds
 &=
          -2 \int_\gamma k_{ss}F\,ds - \int_\gamma k^3F\, ds
\\
 &=
        - 2\int_\gamma k_{ss}^2\,ds
	+ \int_\gamma k_{ss}k^3\,ds
	+ \int_\gamma k^6\, ds
\\
 &=
        - 2\int_\gamma k_{ss}^2\,ds
	- 3\int_\gamma k_s^2k^2\,ds
	+ \int_\gamma k^6\, ds\,,
\end{align*}
as required.
\end{proof}

\section{Fundamental estimates and contraction to a point}
\label{sect:Estimates}

The main difficulty is in obtaining control over the compact case.
We include also the cocompact case for contrast, in Section \ref{sect:cocompact}.

Our main goal is to show that the flow terminates in a point.
This is a crucial step: identification of the final point allows us to remove
any kind of `modulo translation' proviso from our blowup analysis.
For this we need an initial condition on the curvature.

We begin by giving a sufficient condition for the boundedness of the $L^2$-norm of curvature on a bounded time interval.

\begin{lem}\label{LMkinl2}
Suppose $\gamma:I\times[0,T_0]\rightarrow\R^2$ is a compact Chen flow. Assume that there exists a $C>0$ and $D<\infty$ such that
\begin{equation}
\label{EQlnotzero}
L(t) \ge C\quad \text{ for all }\quad t\in[0,T_0]
\end{equation}
and
\begin{equation}
\label{EQksnotinfty}
\vn{k_s}_2^2 \le D\quad \text{ for all }\quad t\in[0,T_0]
\,.
\end{equation}
Then
\[
	\int_\gamma k^2\,ds
	 \le
 	\int_\gamma k^2\, ds\bigg|_{t=0}
	e^{R(C,D,\omega)t}
\]
where $R(C,D,\omega) = \frac{L_0^2}{4\omega^2\pi^2} D^2 + 4
\frac{L_0^\frac12}{(2\omega\pi)^\frac12} D^\frac32 + \frac{12\omega\pi}{C} D +
4\frac{(2\omega\pi)^\frac52}{C^\frac52} D^\frac12 +
\frac{16\omega^4\pi^4}{C^4}$.
\end{lem}
\begin{proof}
We first express four powers of $k$ in $\vn{k}_6^6$ in terms of $(k-\k)$ and $\k$:
\begin{equation}
\label{EQexpando}
\begin{split}
\int_\gamma k^6\,ds
	&= \int_\gamma k^2(k-\k)^4\,ds
	 + 4\k\int_\gamma k^2(k-\k)^3\,ds
\\&\qquad
	 + 6\k^2\int_\gamma k^2(k-\k)^2\,ds
	 + 4\k^3\int_\gamma k^2(k-\k)\,ds
	 +  \k^4\int_\gamma k^2\,ds
\\
	&= \int_\gamma k^2(k-\k)^4\,ds
	 + \frac{8\omega\pi}{L}\int_\gamma k^2(k-\k)^3\,ds
\\&\qquad
	 + \frac{24\omega^2\pi^2}{L^2}\int_\gamma k^2(k-\k)^2\,ds
	 + \frac{32\omega^3\pi^3}{L^3}\int_\gamma k^2(k-\k)\,ds
	 + \frac{16\omega^4\pi^4}{L^4}\int_\gamma k^2\,ds
\,.
\end{split}
\end{equation}
The Poincar\'e inequality implies
\begin{equation}
\label{EQlinfkoscks}
\vn{k-\k}_\infty^2 \le \frac{L}{2\omega\pi} \int_\gamma k_s^2\,ds
\,.
\end{equation}
Combining \eqref{EQexpando} with \eqref{EQlinfkoscks} we have
\begin{equation}
\label{EQest1}
\begin{split}
\int_\gamma k^6\,ds
	&\le \bigg(
		\frac{L^2}{4\omega^2\pi^2} \vn{k_s}_2^4
	 + \frac{8\omega\pi}{L}
	   \frac{L^\frac32}{(2\omega\pi)^\frac32} \vn{k_s}_2^3
\\&\qquad
	 + \frac{24\omega^2\pi^2}{L^2}
	   \frac{L}{2\omega\pi} \vn{k_s}_2^2
	 + \frac{32\omega^3\pi^3}{L^3}
	   \frac{L^\frac12}{(2\omega\pi)^\frac12} \vn{k_s}_2
	 + \frac{16\omega^4\pi^4}{L^4}
		\bigg)\int_\gamma k^2\,ds
\\
	&\le \bigg(
		\frac{L^2}{4\omega^2\pi^2} \vn{k_s}_2^4
	 + 4
	   \frac{L^\frac12}{(2\omega\pi)^\frac12} \vn{k_s}_2^3
\\&\qquad
	 + \frac{12\omega\pi}{L}
	   \vn{k_s}_2^2
	 + 4\frac{(2\omega\pi)^\frac52}{L^\frac52}
	   \vn{k_s}_2
	 + \frac{16\omega^4\pi^4}{L^4}
		\bigg)\int_\gamma k^2\,ds
\,.
\end{split}
\end{equation}
Applying our hypotheses \eqref{EQlnotzero} and \eqref{EQksnotinfty} to the
estimate \eqref{EQest1} (as well as using monotonicity of length, which holds
without special hypotheses) we find
\begin{equation}
\label{EQest2}
\begin{split}
\int_\gamma k^6\,ds
	&\le \bigg(
		\frac{L_0^2}{4\omega^2\pi^2} D^2
	 + 4 \frac{L_0^\frac12}{(2\omega\pi)^\frac12} D^\frac32
\\&\qquad
	 + \frac{12\omega\pi}{C} D
	 + 4\frac{(2\omega\pi)^\frac52}{C^\frac52} D^\frac12
	 + \frac{16\omega^4\pi^4}{C^4}
		\bigg)\int_\gamma k^2\,ds
\\
	&= R(C,D,\omega)\int_\gamma k^2\,ds
\,,
\end{split}
\end{equation}
where $R(C,D,\omega)$ is as in the statement of the Lemma.
Note that $R$ is independent of time.

Then from the evolution equation (Lemma \ref{LMl2kevo}), we find
\[
	\frac{d}{dt}\int_\gamma k^2\,ds
	 \le
	R(C,D,\omega) \int_\gamma k^2\, ds
\,.
\]
Therefore we have
\[
	\int_\gamma k^2\,ds
	 \le
 	\int_\gamma k^2\, ds\bigg|_{t=0}
	e^{R(C,D,\omega)t}
\]
as required.
\end{proof}

Now let us show that the fourth power of length is convex under a scale-invariant initial condition.

\begin{prop}\label{PRl4convex}
Suppose $\gamma:I\times[0,T)\rightarrow\R^2$ is a compact Chen flow with initial data $\gamma_0:I\rightarrow\R^2$ being an immersed $\omega$-circle satisfying
\[
	\SQ(0) = L^3\int_\gamma k_s^2 + k^4\,ds\bigg|_{t=0} < 16\omega^4\pi^4 + \varepsilon_1
\]
where $\varepsilon_1$ is the universal constant depending only on $\omega$ from \eqref{EQseteps1}.

Then the length functional $L:[0,T)\rightarrow\R$ is convex, and there exists an $\varepsilon_2>0$ such that the estimate
\begin{equation}
\label{EQsql1}
\SQ(t_2) + \int_{t_1}^{t_2} 3L^2\vn{k_s}_2^4 + \varepsilon_2 L^3\vn{k_{s^3}}_2^2\,dt \le \SQ(t_1)
\end{equation}
holds.
In particular, $\SQ(t_2) < \SQ(t_1)$ for all $t_1<t_2$ with equality if and
only if $\gamma_{t_1}$ is a standard round $\omega$-circle.
\end{prop}
\begin{proof}
First, we estimate the last term on the right of \eqref{EQl4evo} by
\begin{align*}
	L\int_\gamma k^8 \,ds - \Big(\int_\gamma k^4\,ds\Big)^2
	&= L\int_\gamma k^4
		\bigg( k^4 - \frac1L\int_\gamma k^4\,ds\bigg)
		\,ds
\\
	&= L\int_\gamma
		\bigg( k^4 - \frac1L\int_\gamma k^4\,ds\bigg)^2
		\,ds
\\
	&\qquad + \int_\gamma k^4\,ds\int_\gamma
		\bigg( k^4 - \frac1L\int_\gamma k^4\,ds\bigg)
		\,ds
\\
	&= L\int_\gamma
		\bigg( k^4 - \frac1L\int_\gamma k^4\,ds\bigg)^2
		\,ds
\\
	&\le \frac{4L^3}{\omega^2\pi^2} \int_\gamma k_s^2k^6\, ds
\end{align*}
so that we have
\begin{equation}
\label{EQest1L4}
\begin{split}
-(L^4)''(t)
&\le
	- 8L^3 \int_\gamma
		k_{s^3}^2
		\,ds
	+ 60L^3 \int_\gamma
		k_s^4
		\,ds
	- 12L^2\vn{k_s}_2^4
\\&\qquad
	- 80L^3 \int_\gamma
		k_{ss}^2k^2
		\,ds
	- 88L^3 \int_\gamma
		k_s^2k^4
		\,ds
	- 24L^2\vn{k_s}_2^2\vn{k}_4^4
	+ 48\frac{L^5}{\omega^2\pi^2}
		\int_\gamma k_s^2k^6\, ds
\,.
\end{split}
\end{equation}
To deal with the last term on the right of \eqref{EQest1L4}, we split it up and
absorb into the three other terms on the right, with a small remaining piece
that will be absorbed into the first term on the right.

The first kind of estimate we use is the following:
\begin{equation}
\label{EQest25L4}
\begin{split}
	\frac{L^2}{\omega^2\pi^2}\int_\gamma k_s^2k^6\, ds
	&= \frac{L^2}{\omega^2\pi^2}\int_\gamma k_s^2k^4((k-\k)^2 + 2k\k - \k^2)\, ds
	\\
	&= \frac{L^2}{\omega^2\pi^2}\int_\gamma k_s^2k^4(k-\k)^2\, ds
	 + \frac{L^2}{\omega^2\pi^2}2\k\int_\gamma k_s^2k^4(k-\k+\k)\, ds
	 - \frac{L^2}{\omega^2\pi^2}\k^2\int_\gamma k_s^2k^4\, ds
	\\
	&= \frac{L^2}{\omega^2\pi^2}\int_\gamma k_s^2k^4(k-\k)^2\, ds
	 + \frac{L^2}{\omega^2\pi^2}2\k\int_\gamma k_s^2k^4(k-\k)\, ds
	 + \frac{L^2}{\omega^2\pi^2}\k^2\int_\gamma k_s^2k^4\, ds
	\\
	&= \frac{L^2}{\omega^2\pi^2}\int_\gamma k_s^2k^4(k-\k)^2\, ds
	 + 4\frac{L}{\omega\pi}\int_\gamma k_s^2k^4(k-\k)\, ds
	 + 4\int_\gamma k_s^2k^4\, ds
\,.
\end{split}
\end{equation}
Combining this with \eqref{EQest1L4} we find
\begin{equation}
\label{EQest2L4}
\begin{split}
-(L^4)''(t)
&\le
	- 8L^3 \int_\gamma
		k_{s^3}^2
		\,ds
	+ 60L^3 \int_\gamma
		k_s^4
		\,ds
	- 12L^2\vn{k_s}_2^4
\\&\qquad
	+  22\frac{L^5}{\omega^2\pi^2}\int_\gamma k_s^2k^4(k-\k)^2\, ds
	+ 88\frac{L^4}{\omega\pi}\int_\gamma k_s^2k^4(k-\k)\, ds
\\&\qquad
	- 80L^3 \int_\gamma
		k_{ss}^2k^2
		\,ds
	- 24L^2\vn{k_s}_2^2\vn{k}_4^4
	+ 26\frac{L^5}{\omega^2\pi^2}
		\int_\gamma k_s^2k^6\, ds
\,.
\end{split}
\end{equation}
The next step is to use the following equality (which continues from the last estimate in \eqref{EQest25L4})
\begin{equation}
\label{EQest3L4}
\begin{split}
	\frac{L^2}{\omega^2\pi^2}\int_\gamma k_s^2k^6\, ds
	&=
	  4\int_\gamma k_s^2k^4\, ds
	 + \frac{L^2}{\omega^2\pi^2}\int_\gamma k_s^2k^4(k-\k)^2\, ds
	 + 4\frac{L}{\omega\pi}\int_\gamma k_s^2k^4(k-\k)\, ds
\\
	&=
	  4\int_\gamma k_s^2(k^4-\ol{k^4}+\ol{k^4})\, ds
	 + \frac{L^2}{\omega^2\pi^2}\int_\gamma k_s^2k^4(k-\k)^2\, ds
	 + 4\frac{L}{\omega\pi}\int_\gamma k_s^2k^4(k-\k)\, ds
\\
	&=
	   \frac{4}{L}\vn{k}_4^4\vn{k_s}_2^2
\\&\qquad
	 + 4\int_\gamma k_s^2(k^4-\ol{k^4})\, ds
	 + \frac{L^2}{\omega^2\pi^2}\int_\gamma k_s^2k^4(k-\k)^2\, ds
	 + 4\frac{L}{\omega\pi}\int_\gamma k_s^2k^4(k-\k)\, ds
\,.
\end{split}
\end{equation}
Combining \eqref{EQest3L4} with \eqref{EQest2L4} yields
\begin{equation}
\label{EQest4L4}
\begin{split}
-(L^4)''(t)
&\le
	- 8L^3 \int_\gamma
		k_{s^3}^2
		\,ds
	+ 60L^3 \int_\gamma
		k_s^4
		\,ds
	- 12L^2\vn{k_s}_2^4
\\&\qquad
	+ 28\frac{L^5}{\omega^2\pi^2}\int_\gamma k_s^2k^4(k-\k)^2\, ds
	+ 122\frac{L^4}{\omega\pi}\int_\gamma k_s^2k^4(k-\k)\, ds
	+ 24L^3\int_\gamma k_s^2(k^4-\ol{k^4})\, ds
\\&\qquad
	- 80L^3 \int_\gamma
		k_{ss}^2k^2
		\,ds
	+ 20\frac{L^5}{\omega^2\pi^2}
		\int_\gamma k_s^2k^6\, ds
\,.
\end{split}
\end{equation}
We now perform our last manipulation of the term $\int_\gamma k_s^2k^6\, ds$:
\begin{align*}
	\frac{L^2}{\omega^2\pi^2}
		\int_\gamma k_s^2k^6\, ds
&=
	\frac{L^2}{\omega^2\pi^2}
		\int_\gamma k_s^2k^2((k-\k)^4 + 4k^3\k - 6k^2\k^2 + 4k\k^3 - \k^4)\, ds
\\
&=
	\frac{L^2}{\omega^2\pi^2}
		\int_\gamma k_s^2k^2(k-\k)^4\, ds
	+ \frac{L^2}{\omega^2\pi^2}
		\int_\gamma k_s^2k^2(4\k ((k-\k)^3 + 3k^2\k - 3k\k^2 + \k^3))\, ds
\\&\qquad
	+ \frac{L^2}{\omega^2\pi^2}
		\int_\gamma k_s^2k^2( - 6k^2\k^2 + 4k\k^3 - \k^4)\, ds
\\
&=
	\frac{L^2}{\omega^2\pi^2}
		\int_\gamma k_s^2k^2(k-\k)^4\, ds
	+ 8\frac{L}{\omega\pi}
		\int_\gamma k_s^2k^2(k-\k)^3\, ds
\\&\qquad
	+ \frac{L^2}{\omega^2\pi^2}
		\int_\gamma k_s^2k^2( 6\k^2((k-\kv)^2 + 2k\k - \k^2) - 8k\k^3 + 3\k^4)\, ds
\\
&=
	\frac{L^2}{\omega^2\pi^2}
		\int_\gamma k_s^2k^2(k-\k)^4\, ds
	+ 8\frac{L}{\omega\pi}
		\int_\gamma k_s^2k^2(k-\k)^3\, ds
	+ 24
		\int_\gamma k_s^2k^2(k-\kv)^2\, ds
\\&\qquad
	+ \frac{L^2}{\omega^2\pi^2}
		\int_\gamma k_s^2k^2( 4k\k^3 - 3\k^4)\, ds
\\
&=
	\frac{L^2}{\omega^2\pi^2}
		\int_\gamma k_s^2k^2(k-\k)^4\, ds
	+ 8\frac{L}{\omega\pi}
		\int_\gamma k_s^2k^2(k-\k)^3\, ds
	+ 24
		\int_\gamma k_s^2k^2(k-\kv)^2\, ds
\\&\qquad
	+ \frac{L^2}{\omega^2\pi^2}
		\int_\gamma k_s^2k^2( 4\k^3((k-\k) + \k) - 3\k^4)\, ds
\\
&=
	\frac{L^2}{\omega^2\pi^2}
		\int_\gamma k_s^2k^2(k-\k)^4\, ds
	+ 8\frac{L}{\omega\pi}
		\int_\gamma k_s^2k^2(k-\k)^3\, ds
	+ 24
		\int_\gamma k_s^2k^2(k-\kv)^2\, ds
\\&\qquad
	+ 32\frac{\omega\pi}{L}
		\int_\gamma k_s^2k^2(k-\k)\, ds
	+ 16\frac{\omega^2\pi^2}{L^2}
		\int_\gamma k_s^2k^2\, ds
\\
&\le
	\frac{L^2}{\omega^2\pi^2}
		\int_\gamma k_s^2k^2(k-\k)^4\, ds
	+ 8\frac{L}{\omega\pi}
		\int_\gamma k_s^2k^2(k-\k)^3\, ds
	+ 24
		\int_\gamma k_s^2k^2(k-\kv)^2\, ds
\\&\qquad
	+ 32\frac{\omega\pi}{L}
		\int_\gamma k_s^2k^2(k-\k)\, ds
	+ 4
		\int_\gamma (k_{ss}k + 2k_s^2)^2\,ds
\\
&\le
	\frac{L^2}{\omega^2\pi^2}
		\int_\gamma k_s^2k^2(k-\k)^4\, ds
	+ 8\frac{L}{\omega\pi}
		\int_\gamma k_s^2k^2(k-\k)^3\, ds
	+ 24
		\int_\gamma k_s^2k^2(k-\kv)^2\, ds
\\&\qquad
	+ 32\frac{\omega\pi}{L}
		\int_\gamma k_s^2k^2(k-\k)\, ds
	+ 4
		\int_\gamma k_{ss}^2k^2\,ds
	+ \frac{32}{3}
		\int_\gamma k_s^4\,ds
\,.
\end{align*}
Inserting this estimate into \eqref{EQest4L4} we find
\begin{equation}
\label{EQest5L4}
\begin{split}
-(L^4)''(t)
&\le
	- 8L^3 \int_\gamma
		k_{s^3}^2
		\,ds
	- 12L^2\vn{k_s}_2^4
	+ \frac{820}{3}L^3
		\int_\gamma k_s^4\,ds
\\&\qquad
	+ 24L^3\int_\gamma k_s^2(k^4-\ol{k^4})\, ds
	+ 28\frac{L^5}{\omega^2\pi^2}\int_\gamma k_s^2k^4(k-\k)^2\, ds
	+ 122\frac{L^4}{\omega\pi}\int_\gamma k_s^2k^4(k-\k)\, ds
\\&\qquad
	+ 20\frac{L^5}{\omega^2\pi^2}
		\int_\gamma k_s^2k^2(k-\k)^4\, ds
	+ 160\frac{L^4}{\omega\pi}
		\int_\gamma k_s^2k^2(k-\k)^3\, ds
	+ 480L^3
		\int_\gamma k_s^2k^2(k-\kv)^2\, ds
\\&\qquad
	+ 640L^2\omega\pi
		\int_\gamma k_s^2k^2(k-\k)\, ds
\,.
\end{split}
\end{equation}
Now our goal is to absorb terms three through to ten (on the RHS) into the first term.
Suppose that
\begin{equation}
\label{EQtopreserveL4}
	\SQ(t) = L^3\int_\gamma k_s^2 + k^4\,ds \le \SQ(0) = \varepsilon_0 + 16\pi^4\omega^4
\,.
\end{equation}
Since $L \le \frac{L^2}{4\omega^2\pi^2} \int_\gamma k^2\,ds \le \frac{L^2}{4\omega^2\pi^2} \sqrt{L}\bigg(\int_\gamma k^4\,ds\bigg)^\frac12$, we have $L^3\int_\gamma k^4\,ds \ge 16\pi^2\omega^4$ and so
\[
	L^3\int_\gamma k_s^2\,ds \le \varepsilon_0\,.
\]
We work only on an interval $[0,t_0)$ during which \eqref{EQtopreserveL4} holds.
Note that it holds at initial time; in fact, we will show $\SQ'(0) < 0$, so $t_0>0$.
Then we will have that on $[0,t_0)$, $\SQ'(t) < 0$ and so we may take $t_0=T$
and the proof is finished.

We go through each term from \eqref{EQest5L4} in turn.
\begin{equation}
\label{EQterm1L4}
	\begin{split}
	  \frac{820}{3}L^3
		\int_\gamma k_s^4\,ds
	&\le \frac{820}{3}\varepsilon_0 \vn{k_s}_\infty^2
	\\
	&\le \frac{820}{3}\varepsilon_0 \frac{L}{2\omega\pi}\vn{k_{ss}}_2^2
	 \le \varepsilon_0L^3 \frac{820}{24\omega^3\pi^3}\vn{k_{s^3}}_2^2
	\end{split}
\end{equation}
\begin{equation}
\label{EQterm2L4}
	\begin{split}
	  24L^3\int_\gamma k_s^2(k^4-\ol{k^4})\, ds
	&\le 24L^3\int_\gamma k_s^2\,ds\vn{k^4-\ol{k^4}}_\infty
	\\
	&\le 96L^3\int_\gamma k_s^2\,ds\int_\gamma |k_sk^3|\,ds
	\\
	&\le 96L^3\int_\gamma k_s^2\,ds\bigg(\int_\gamma k_s^4\,ds\bigg)^\frac14\bigg(\int_\gamma k^4\,ds\bigg)^\frac34
	\\
	&\le  L^3\int_\gamma k_s^4\,ds
	   + \frac{3\cdot 96^\frac43}{4(4)^\frac13}L^3\int_\gamma k^4\,ds\bigg(\int_\gamma k_s^2\,ds\bigg)^\frac43
	\\
	&\le  L^3\int_\gamma k_s^4\,ds
	   + \frac{3\cdot 96^\frac43}{4(4)^\frac13}\SQ(0)L^{-1}\varepsilon_0^{\frac13}\int_\gamma k_s^2\,ds
	\\
	&\le  L^3\int_\gamma k_s^4\,ds
	   + L^{3}\varepsilon_0^{\frac13}\frac{3\cdot 96^\frac43}{64\omega^4\pi^4(4)^\frac13}\SQ(0)\vn{k_{s^3}}_2^2
	\end{split}
\end{equation}
\begin{equation}
\label{EQterm3L4}
	\begin{split}
	  28\frac{L^5}{\omega^2\pi^2}\int_\gamma k_s^2k^4(k-\k)^2\, ds
	&\le
	  28\frac{L^5}{\omega^2\pi^2}\vn{k_s}_\infty^2\vn{k-\k}_\infty^2\int_\gamma k^4\, ds
	\\
	&\le
	  28\frac{L^5}{\omega^2\pi^2}\frac{L^3}{8\omega^3\pi^3}\vn{k_{s^3}}_2^2\frac{L}{2\omega\pi}\vn{k_s}_2^2L^{-3}\SQ(0)
	\\
	&\le
	  L^3\varepsilon_0 \frac{7}{4\omega^4\pi^4}\SQ(0)\vn{k_{s^3}}_2^2
	\end{split}
\end{equation}
\begin{equation}
\label{EQterm4L4}
	\begin{split}
	 122\frac{L^4}{\omega\pi}\int_\gamma k_s^2k^4(k-\k)\, ds
	&\le
	 122\frac{L^4}{\omega\pi}\vn{k_s}_\infty^2\vn{k-\k}_\infty\int_\gamma k^4\, ds
	\\
	&\le
	 122\frac{L^4}{\omega\pi}\frac{L^3}{8\omega^3\pi^3}\vn{k_{s^3}}_2^2\frac{L^\frac12}{(2\omega\pi)^\frac12}\vn{k_s}_2L^{-3}\SQ(0)
	\\
	&\le
	 L^3\varepsilon_0^\frac12 \frac{61}{2^\frac12(\omega\pi)^\frac52}\SQ(0)\vn{k_{s^3}}_2^2
	\end{split}
\end{equation}
\begin{equation}
\label{EQterm5L4}
	\begin{split}
	  20\frac{L^5}{\omega^2\pi^2}
		\int_\gamma k_s^2k^2(k-\k)^4\, ds
	&\le
	  20\frac{L^5}{\omega^2\pi^2}
		\vn{k-\k}_\infty^4
		L^\frac12\vn{k}_4^2
		\vn{k_s}_\infty^2
	\\
	&\le
	  20\frac{L^5}{\omega^2\pi^2}
		L^{-1}\sqrt{\SQ(0)}
		\frac{L^2}{4\omega^2\pi^2}\vn{k_s}_2^4
		\frac{L^3}{8\omega^3\pi^3}\vn{k_{s^3}}_2^2
	\\
	&\le
	  L^3\varepsilon_0^2 \frac{5}{8\omega^7\pi^7}
		\sqrt{\SQ(0)}
		\vn{k_{s^3}}_2^2
	\end{split}
\end{equation}
\begin{equation}
\label{EQterm6L4}
	\begin{split}
	  160\frac{L^4}{\omega\pi}
		\int_\gamma k_s^2k^2(k-\k)^3\, ds
	&\le  160\frac{L^4}{\omega\pi}
		L^{-1}\sqrt{\SQ(0)}
		\vn{k-\k}_\infty^3
		\frac{L^3}{8\omega^3\pi^3}\vn{k_{s^3}}_2^2
	\\
	&\le  160\frac{L^3}{\omega\pi}
		\sqrt{\SQ(0)}
		\frac{L^\frac32}{(2\omega\pi)^\frac32}\vn{k_s}_2^3
		\frac{L^3}{8\omega^3\pi^3}\vn{k_{s^3}}_2^2
	\\
		&\le  L^3\varepsilon_0^\frac32 \frac{10}{2^\frac12(\omega\pi)^\frac{11}2}
		\sqrt{\SQ(0)}
		\vn{k_{s^3}}_2^2
	\end{split}
\end{equation}
\begin{equation}
\label{EQterm7L4}
	\begin{split}
	  480L^3
		\int_\gamma k_s^2k^2(k-\kv)^2\, ds
	&\le 480L^3
		L^{-1}\sqrt{\SQ(0)}
		\frac{L}{2\omega\pi}\vn{k_s}_2^2
		\frac{L^3}{8\omega^3\pi^3}\vn{k_{s^3}}_2^2
	\\
	&\le L^3\varepsilon_0
		\sqrt{\SQ(0)}
		\frac{30}{\omega^4\pi^4}
		\vn{k_{s^3}}_2^2
	\end{split}
\end{equation}
\begin{equation}
\label{EQterm8L4}
	\begin{split}
	  640L^2\omega\pi
		\int_\gamma k_s^2k^2(k-\k)\, ds
	&\le  640L^2{\omega\pi}
		L^{-1}\sqrt{\SQ(0)}
		\vn{k-\k}_\infty
		\vn{k_s}_\infty^2
	\\
	&\le  640L^2\omega\pi
		L^{-1}\sqrt{\SQ(0)}
		\frac{L^\frac12}{(2\omega\pi)^\frac12}\vn{k_s}_2
		\frac{L^3}{8\omega^3\pi^3}\vn{k_{s^3}}_2^2
	\\
	&\le  640L^3
		\sqrt{\SQ(0)}
		\frac{L^{-\frac12}}{(2\omega\pi)^\frac12}L^{-\frac32}\varepsilon_0^\frac12
		\frac{L^2}{16\omega^3\pi^3}\vn{k_{s^3}}_2^2
	\\
	&\le  L^3\varepsilon_0^\frac12
		\frac{40}{2^\frac12(\omega\pi)^\frac72}
		\sqrt{\SQ(0)}
		\vn{k_{s^3}}_2^2
	\end{split}
\end{equation}
Combining estimates \eqref{EQterm1L4}--\eqref{EQterm8L4} with \eqref{EQest5L4} we find
\begin{equation}
\label{EQest6L4}
\begin{split}
-(L^4)''(t)
&\le
	- 12L^2\vn{k_s}_2^4
	+
	\bigg(-8
	+ \varepsilon_0\frac{205}{6\omega^3\pi^3}
	+
	  \varepsilon_0\frac{1}{8\omega^3\pi^3}
	+ \varepsilon_0^{\frac13}\frac{3\cdot 96^\frac43}{64\omega^4\pi^4(4)^\frac13}\SQ(0)
\\&\qquad
	+ \varepsilon_0 \frac{7}{4\omega^4\pi^4}\SQ(0)
	+ \varepsilon_0^\frac12 \frac{61}{2^\frac12(\omega\pi)^\frac52}\SQ(0)
	+ \varepsilon_0^2 \frac{5}{8\omega^7\pi^7} \sqrt{\SQ(0)}
	+ \varepsilon_0^\frac32 \frac{10}{2^\frac12(\omega\pi)^\frac{11}2}\sqrt{\SQ(0)}
\\&\qquad
	+ \varepsilon_0\sqrt{\SQ(0)}\frac{30}{\omega^4\pi^4}
	+ \varepsilon_0^\frac12\frac{40}{2^\frac12(\omega\pi)^\frac72}\sqrt{\SQ(0)}
	\bigg)L^3 \int_\gamma
		k_{s^3}^2
		\,ds
\,.
\end{split}
\end{equation}
Since $\SQ(0) \le \varepsilon_0 + 16\pi^4\omega^4$, setting
\begin{align*}
	R_0(\varepsilon_0) &=
	  \varepsilon_0\frac{205}{6\omega^3\pi^3}
	+
	  \varepsilon_0\frac{1}{8\omega^3\pi^3}
	+ \varepsilon_0^{\frac13}\frac{3\cdot 96^\frac43}{64\omega^4\pi^4(4)^\frac13}(\varepsilon_0 + 16\pi^4\omega^4)
\\&\qquad
	+ \varepsilon_0 \frac{7}{4\omega^4\pi^4}(\varepsilon_0 + 16\pi^4\omega^4)
	+ \varepsilon_0^\frac12 \frac{61}{2^\frac12(\omega\pi)^\frac52}(\varepsilon_0 + 16\pi^4\omega^4)
	+ \varepsilon_0^2 \frac{5}{8\omega^7\pi^7} \sqrt{(\varepsilon_0 + 16\pi^4\omega^4)}
\\&\qquad
	+ \varepsilon_0^\frac32 \frac{10}{2^\frac12(\omega\pi)^\frac{11}2}\sqrt{(\varepsilon_0 + 16\pi^4\omega^4)}
	+ \varepsilon_0\sqrt{(\varepsilon_0 + 16\pi^4\omega^4)}\frac{30}{\omega^4\pi^4}
	+ \varepsilon_0^\frac12\frac{40}{2^\frac12(\omega\pi)^\frac72}\sqrt{(\varepsilon_0 + 16\pi^4\omega^4)}
\end{align*}
we use \eqref{EQest6L4} to find
\begin{equation}
\label{EQest7L4}
-(L^4)''(t)
\le
	- 12L^2\vn{k_s}_2^4
	+ (-8 + R_0(\varepsilon_0))L^3 \int_\gamma
		k_{s^3}^2
		\,ds
\,.
\end{equation}
Clearly $R_0$ tends monotonically to zero as $\varepsilon_0$ goes to zero, and so there exists an $\varepsilon_1>0$, depending only on $\omega$ (which is a universal constant) such that
\begin{equation}
\label{EQseteps1}
	R_0(\varepsilon_0) < 8\qquad\text{for all}\quad \varepsilon_0 < \varepsilon_1
\,.
\end{equation}
Now we finish our estimate.
Assuming the initial condition
\[
	\SQ(0) < 16\pi^4\omega^4 + \varepsilon_1\,,
\]
we see that $\SQ'(0) < 0$ and in fact $\SQ'(0) = 0$ if and only if $\gamma_0$
is a multiply-covered circle, in which case we have in fact $\SQ(0) =
16\pi^4\omega^4 = \SQ(t)$ for all $t\in[0,T)$.

So, as $\SQ'(0) < 0$, smoothness of the solution implies that there exists a maximal $t_0$ such that $\SQ(t) \le \SQ(0)$ for all $t\in[0,t_0)$.
In fact, we must have $t_0 = T$ since, if not, we still have $\SQ(t_0) \le \SQ(0)$ by taking a limit (as $t_0 < T$ this limit exists).
Then applying again the estimate \eqref{EQest7L4} yields $\SQ'(t_0) < 0$, or $\gamma_{t_0}$ is again a standard round $\omega$-circle.
Either case is a contradiction with the maximality of $t_0$, and so we must have $t_0=T$.
\end{proof}

A landmark observation in the field is that given by the interpolation method
of Dziuk-Kuwert-Sch\"atzle \cite{DKS2002}.
As commented in \cite[p. 1236]{DKS2002} this applies for flows of curves that
have up to a fifth power of curvature in the speed (but not including the fifth
power).
Here, our flow has a third power of curvature (as does the elastic flow
considered in \cite{DKS2002}) and so we may use the same techniques as in
\cite{DKS2002} to obtain the following result.
Note that we always have $T<\infty$ in the compact case of Chen's flow (recall
the estimate \eqref{FT}) so we do not require this additional hypothesis here.

\begin{thm}\label{TMblowup}
Suppose $\gamma:I\times[0,T)\rightarrow\R^2$ is a compact Chen flow.
Then
\begin{equation}
\label{EQblowupest}
	\int_\gamma k^2\,ds \ge c_B(T-t)^{-\frac14}\,,
\end{equation}
where $c_B$ is a constant depending only on $\omega$.

If $\gamma:I\times[0,T)\rightarrow\R^2$ is a cocompact Chen flow, then
\eqref{EQblowupest} holds only when the maximal time of smooth existence $T$ is
finite.
\end{thm}

We are now in a position to prove that the flow contracts to a point.

\begin{thm}\label{TMfinalpoint}
Suppose $\gamma:I\times[0,T)\rightarrow\R^2$ is a compact Chen flow with initial data $\gamma_0:I\rightarrow\R^2$ being an immersed $\omega$-circle satisfying
\begin{equation}
\label{EQhypothm}
	\SQ(0) = L^3\int_\gamma k_s^2 + k^4\,ds\bigg|_{t=0} < 16\omega^4\pi^4 + \varepsilon_1
\end{equation}
where $\varepsilon_1$ is the universal constant depending only on $\omega$ from \eqref{EQseteps1}.

Then $L(t) \searrow 0$ as $t\nearrow T$, and $\gamma(I,t)\rightarrow\{\SO\}$, where $\SO\in\R^2$, in the Hausdorff metric on subsets of $\R^2$.
\end{thm}
\begin{proof}
We have that $L'(t) \le 0$ by Lemma \ref{LMlengthevo}.
We assume for the purpose of contradiction that $L(t)\rightarrow C$ where $C>0$. Then $L(t) \ge C$ for all $t\in[0,T_0]$, where $T_0<T$.
The hypothesis \eqref{EQhypothm} implies that
\[
	L^3\int_\gamma k_s^2\, ds \le \SQ(0)
\]
for all $t\in[0,T_0]$.
Since $L(t) \ge C$ this implies
\[
	\int_\gamma k_s^2\, ds \le \SQ(0)C^{-3} =: D
\,.
\]
Therefore we can apply Lemma \ref{LMkinl2} to obtain
\begin{equation}
\label{EQcontra}
	\int_\gamma k^2\,ds
	 \le
 	\int_\gamma k^2\, ds\bigg|_{t=0}
	e^{R(C,D,\omega)T}
	\le C_0
\end{equation}
where $C_0$ depends only on $\omega$, $C$ and $D$.
Note that there is no dependence on $T_0$.

The estimate from Theorem \ref{TMblowup} implies
\[
	\int_\gamma k^2\,ds \ge c_B(T-t)^{-\frac14}
	\ge 2C_0\,,
\]
for $t\in[T-\frac{c_B^4}{16C_0^4},T)$.
This is a contradiction with \eqref{EQcontra}.

Therefore $L(t)\searrow0$ as $t\nearrow T$.
Since $\gamma:I\times[0,T)\rightarrow\R^2$ is a smooth family of immersed
$\omega$-circles, this implies that the images $\gamma(I,t)$ converge in the
Hausdorff metric on subsets of $\R^2$ to a point.
\end{proof}

\section{The cocompact case}
\label{sect:cocompact}

In the cocompact case, a much simpler estimate holds for the functional $\SR(t) = L\vn{k}_2^2$.
There is no need for a rescaling, and we enjoy uniform length estimates for the original flow.
In order to illustrate these benefits, let us quickly prove Theorem
\ref{TMcoco2} from the introduction.

\begin{lem}
\label{LMcocolength}
Let $\gamma:\R\times[0,T)\rightarrow\R^2$ be a cocompact Chen flow with initial
data $\gamma_0:I\rightarrow\R^2$ being an immersed interval.
Then
\[
1 \le L(t) \le L(0)\,.
\]
\end{lem}
\begin{proof}
While the flow exists, the cocompactness condition implies $L(t) \ge 1$ (with equality if and only if $\gamma$ is a straight line).
The evolution of $L$ implies
\[
L' = -\int_\gamma k_s^2 + k^4\,ds \le 0
\]
which implies that the length is monotone decreasing, that is, the second estimate.
\end{proof}

\begin{prop}
\label{PNcocosr}
Let $\gamma:\R\times[0,T)\rightarrow\R^2$ be a cocompact Chen flow with initial
data $\gamma_0:I\rightarrow\R^2$ being an immersed interval satisfying $\omega = 0$.

Suppose in addition that
\begin{equation}
\label{EQcoco1}
\SR(0) = L\int_{\gamma} k^2\,ds\bigg|_{t=0} \le \frac34\,.
\end{equation}
Then for all $t\in[0,T)$,
\[
\SR(t) \le \SR(0)e^{-2L_0^{-4}t}\,.
\]
\end{prop}
\begin{proof}
Since
\[
\int_\gamma k^6\,ds \le \vn{k}_\infty^4\vn{k}_2^2
\]
and (note that $\omega = \k = 0$ here)
\[
\vn{k^2}_\infty \le 2\int_\gamma |k_sk|\,ds
 \le 2L^\frac12 \bigg(\int_\gamma k^2k_s^2\,ds\bigg)^\frac12
\,,
\]
we have
\[
\int_\gamma k^6\,ds \le 4L\vn{k}_2^2\,\int_\gamma k^2k_s^2\,ds
\]
This implies
\begin{align*}
(\vn{k}_2^2)'
 &= -2\int_\gamma kk_{s^4}\,ds
 + \int_\gamma 4k^3k_{ss} + 12k_s^2k^2 + 2k^6 + k^3k_{ss} - k^6\,ds
\\
 &= -2\vn{k_{ss}}_2^2
 + \int_\gamma 5k^3k_{ss} + 12k_s^2k^2 + k^6\,ds
\\
 &= -2\vn{k_{ss}}_2^2
 - 3\int_\gamma k^2k_s^2\,ds
 + \int_\gamma k^6\,ds
\\
 &\le -2\vn{k_{ss}}_2^2
 + (4L\vn{k}_2^2 - 3)\int_\gamma k^2k_s^2\,ds
\\
 &\le -2\vn{k_{ss}}_2^2
\,,
\end{align*}
using the hypothesis \eqref{EQcoco1}.
Applying the Poincar\'e inequality, we find
\[
\vn{k}_2^2 \le L^2\vn{k_s}_2^2 \le L^4\vn{k_{ss}}_2^2
\]
which implies
\[
\SR' \le -2L^{-3}\vn{k}_2^2 = -2L^{-4}\SR
     \le -2L_0^{-4}\SR
\]
or
\[
\SR(t) \le \SR(0)e^{-2L_0^{-4}t}
\]
as required.
\end{proof}

Now this combines with Theorem \ref{TMblowup} to give long time existence.

\begin{thm}
Let $\gamma:\R\times[0,T)\rightarrow\R^2$ be a cocompact Chen flow with initial
data $\gamma_0:I\rightarrow\R^2$ being an immersed interval satisfying $\omega = 0$ and \eqref{EQcoco1}.

Then the maximal time of exsitence is infinite ($T=\infty$).
\end{thm}
\begin{proof}
If $T<\infty$ then Theorem \ref{TMblowup} implies that $\vn{k}_2^2$ is unbounded as a function of time.
However the lower bound for length (Lemma \ref{LMcocolength}) and the decay estimate for $\SR$ (Proposition \ref{PNcocosr}) contradict this.
Therefore $T=\infty$.
\end{proof}

The uniform length and $\SR$ estimates imply uniform control on all derivatives of curvature.
The argument is essentially from Dziuk-Kuwert-Sch\"atzle \cite{DKS2002}, and the key evolution equation (which is \cite[equation (3.2)]{DKS2002}) is also used in the proof of Theorem \ref{TMblowup}.
Since, for efficiently, we obtain exponential decay at the same time, we give the argument.

\begin{thm}
\label{TMcoco1}
Let $\gamma:\R\times[0,T)\rightarrow\R^2$ be a cocompact Chen flow with initial
data $\gamma_0:I\rightarrow\R^2$ being an immersed interval satisfying $\omega = 0$ and \eqref{EQcoco1}.

For each $m\in\N_0$ there exist constants $C_m$ and $c_m$ depending only on $\vn{k_{s^m}}_2^2(0)$, $L_0$ and $\SR(0)$ such that
\[
	\int_{\gamma} k_{s^m}^2\,ds \le C_m e^{-c_mt}
	\,,\qquad \text{ for all $t\in[0,\infty)$}\,.
\]
\end{thm}
\begin{proof}
As in Dziuk-Kuwert-Sch\"atzle we have \cite[equation (3.2)]{DKS2002}:
\begin{align*}
	\frac{d}{dt}\int_{\gamma}k_{s^m}^2\,ds
	 + \int_{\gamma}k_{s^{m+2}}^2\,ds
	\le
	 c_{m,1}\bigg(\int_{\gamma}k^2\,ds\bigg)^{2m+5}
\,.
\end{align*}
The uniform length and $\SR$ estimates imply $\vn{k}_2^2 \le L^{-1}\SR \le \SR(0)e^{-2L_0^{-4}t}$, so
\begin{align*}
	\frac{d}{dt}\int_{\gamma}k_{s^m}^2\,ds
	 + \int_{\gamma}k_{s^{m+2}}^2\,ds
	\le
	  c_{m,2}e^{-c_{m,2}t}
\,.
\end{align*}
Now we use the Poincar\'e inequality and the length estimates to find
\begin{align*}
	\frac{d}{dt}\int_{\gamma}k_{s^m}^2\,ds
	\le
	  c_{m,2}e^{-c_{m,3}t}
	-d\int_{\gamma}k_{s^m}^2\,ds
\,.
\end{align*}
This implies that $\vn{k_{s^m}}_2^2 \le \frac{c_{m,2}}{d}e^{-c_{m,3}t}$, as required.
\end{proof}

These estimates imply convergence of the position vector in the smooth topology, completing the proof of Theorem \ref{TMcoco2}.

\begin{cor}
\label{CYcoco}
Let $\gamma:\R\times[0,T)\rightarrow\R^2$ be a cocompact Chen flow with initial
data $\gamma_0:I\rightarrow\R^2$ being an immersed interval satisfying $\omega = 0$ and \eqref{EQcoco1}.

Then there exists a straight line $\gamma_\infty:\R\times[0,T)\rightarrow\R^2$
such that $\gamma\rightarrow \gamma_\infty$ exponentially fast in the $C^\infty$ topology.
\end{cor}
\begin{proof}
The exponential decay (Theorem \ref{TMcoco1}) of curvature and its derivatives imply
\[
	|\gamma_t| \le |k_{ss}| + |k|^3 + 3|k|^2\,|k_s| \le C_0e^{-c_0t}
	\,.
\]
Therefore
\[
	|\gamma(u,t_2) - \gamma(u,t_1)| = \bigg|\int_{t_1}^{t_2}\gamma_t(u,t)\,dt\bigg|
		\le C(e^{-c_0t_1} + e^{-c_0t_2})
\]
and $\gamma$ is convergent to some $\gamma_\infty$ in $C^0$.
The estimates on curvature and its derivatives upgrade this convergence to $C^\infty$.
In fact, since $k$ converges exponentially to zero, $\gamma_\infty$ must be a straight line.
\end{proof}

\begin{rmk}
We do not know if there is any way to determine (using only the initial data)
to which line the flow will converge.
\end{rmk}


\section{Asymptotic behaviour of the compact flow}
\label{sect:Asymptotics}
  
Assume the initial condition \eqref{EQhypothm}.
We rescale the flow $\gamma:I\times[0,T)\rightarrow\R^2$ to $\hat\gamma:I\times[0,\infty)\rightarrow\R^2$ by setting
\[
\hat\gamma(u, t) = (4T)^{-\frac14}e^{t}\Big(\gamma(u,T - Te^{-4t}) - \SO\Big)
\,,
\]
where $\SO$ is the final point for the flow $\gamma$ from Theorem \ref{TMfinalpoint}.
We find
\begin{align*}
\partial_t\hat\gamma
 &= \hat\gamma
 + 4T(4T)^{-\frac14}e^{t-4t}[ -\gamma_{s^4}](u,T-Te^{-ct})
\\
 &= \hat\gamma
 + 4T(4T)^{-1}e^{4t-4t} [ -\rgamma_{\rs^4}](u,T-Te^{-4t})
\,.
\end{align*}
In this way we generate, from the Chen flow $\gamma$, the rescaled flow $\hat\gamma:I\times[0,\infty)\rightarrow\R^2$.
The rescaling satisfies the evolution equation
\[
\hat\gamma_{t} = \hat\gamma - \rgamma_{\rs^4}
               = -(\hat k_{\hat s\hat s} - \hat k^3 - \IP{\rgamma}{\rnu})\hat\nu + (\IP{\rgamma}{\rtau} + 3\rk\rk_{\rs})\rtau
\,,
\]
and has initial data that differs from the initial data of $\gamma$ by scaling and translation; in particular
\[
\rgamma(u,0) = (4T)^{-\frac14}\Big(\gamma(u,0) - \SO\Big)
\,.
\]


The main purpose of performing the rescaling is to examine the asymptotic shape of the final point for the flow $\gamma$.
We first prove that the rescaling successfully yields control on the new length $\rL$.

\begin{lem}
\label{LMrlb}
Suppose $\gamma:I\times[0,T)\rightarrow\R^2$ is a compact Chen flow with
initial data $\gamma_0:I\rightarrow\R^2$ being an immersed $\omega$-circle
satisfying \eqref{EQhypothm}.

Let $\rgamma:I\times[0,\infty)\rightarrow\R^2$ be the rescaling of $\gamma$ around the final point $\SO$ (from Theorem \ref{TMfinalpoint}).
Then
\[
	c_B\sqrt2
	\le
	\rL(t)
	\le
	\frac{(\hat \SQ(0))^\frac12}{2^\frac12 c_B}
\,.
\]
\end{lem}
\begin{proof}
Theorem \ref{TMblowup} implies that
\[
\int_{\gamma} k^2\,ds \ge c_B(T-t)^{-1/4}\,.
\]
In the rescaling, we have
\begin{align*}
\int_{\hat\gamma} \hat k^2\,d\hat s
 &= (4T)^\frac14 e^{-t}
\int_{\gamma} k^2\,ds\bigg|_{T-Te^{-4t}}
\\
 &\ge c_B(4T)^\frac14 e^{-t}
     (T-(T-Te^{-4t}))^{-1/4}
\\
 &=
  c_B(4T)^\frac14 e^{-t}
     (Te^{-4t})^{-1/4}
  =
  c_B\sqrt2
\,.
\end{align*}
This means that a \emph{universal} (that is, not initial-data-dependent) quantum of curvature persists through the rescaling.

Observe that the estimate \eqref{EQsql1} implies, by rescaling, (here $t_1, t_2 \in [0,\infty)$, $t_1 < t_2$ are arbitrary)
\begin{equation}
\label{EQrsqestl1}
\rSQ(t_2) + \int_{t_1}^{t_2} 3\rL^2\vn{\rk_{\rs}}_2^4 + \varepsilon_2 \rL^3\vn{\rk_{{\rs}^3}}_2^2\,dt \le \rSQ(t_1)
\end{equation}
Since (using \eqref{EQrsqestl1} for the last inequality)
\begin{equation}
\label{EQrescaledest1}
	2c_B^2 \le \vn{\hat k}_2^4 \le \hat L\vn{\hat k}_4^4 \le \hat L^{-2}\hat \SQ(t) \le \hat L^{-2}\hat \SQ(0)
\end{equation}
we have the following bound for rescaled length
\[
	\hat L^2 \le \frac{\hat \SQ(0)}{2c_B^2}
	\,.
\]
For the lower bound we will use the evolution equation for rescaled length.

The stated evolution equations (Lemma \ref{LMevo}) hold for a flow with arbitrary speed, including tangential terms (as does the evolution of $\rgamma$).
Using in particular \eqref{EQlengthevo} we find
\begin{align*}
\rL'
	&= \int (\hat k_{\hat s\hat s} - \hat k^3 - \IP{\rgamma}{\rnu})\rk + \rG_{\rs}\,d\rs
\\
	&= \rL - \int_{\rgamma} \rk_{\rs}^2 + \rk^4\,d\rs
	= \rL - \rL^{-3}\rSQ
\,.
\end{align*}
Using again \eqref{EQrescaledest1} we find
\[
	{\hat L}'(t)
	\le \hat L(t) - \frac{2c_B^2}{\hat L}
\]
or
\[
	(\hat L^2)'(t) \le 2\hat L^2(t) - 4c_B^2
	\,.
\]
If
\[
	\hat L(t) < c_B\sqrt2\,,
\]
then $L$ will decrease. Since the inequality is strict, the rescaled length will vanish in finite time.
We already know that the rescaling $\hat\gamma$ exists for all time, so this can't happen.
This means that we have
\[
	\hat L(t) \ge c_B\sqrt2\,,
\]
a uniform bound from below.
\end{proof}

The uniform length bound combined with our estimate for the scale-invariant quantity $\rSQ$ yields eventual uniform positivity of the rescaled curvature $\rk$ and smallness of $\rk_{\rs}$.

\begin{lem}\label{LMkpos}
Suppose $\gamma:I\times[0,T)\rightarrow\R^2$ is a compact Chen flow with
initial data $\gamma_0:I\rightarrow\R^2$ being an immersed $\omega$-circle
satisfying \eqref{EQhypothm}.

Let $\rgamma:I\times[0,\infty)\rightarrow\R^2$ be the rescaling of $\gamma$ around the final point $\SO$ (from Theorem \ref{TMfinalpoint}).
There exists a $\delta_0 = \delta_0(\rSQ(0)) > 0$ and sequence of times $\{t_j\}$, $t_j\rightarrow\infty$, along which
\[
\vn{\rk_{\rs^3}}_2^2(t_j) \searrow0\,,\quad\text{ and }\quad
\rk(u,t_j) \ge \frac{\delta_0}{2}\,,\quad\text{ for all $u\in I$.}
\]
\end{lem}
\begin{proof}
The average of $\rk$ can be estimated from below by Lemma \ref{LMrlb}:
\[
\bar{\rk} = \frac{2\omega\pi}{\rL} \ge \frac{2\omega\pi c_B\sqrt2}{\sqrt{\rSQ(0)}} =: \delta_0\,.
\]
Therefore, since $\rgamma$ is smooth and closed, at each time $t\in[0,\infty)$ there exists at least one $u_0\in I$ such that
\[
\rk(u_0,t) \ge \delta_0
\,.
\]

Now estimate \eqref{EQrsqestl1} and Lemma \ref{LMrlb} imply that along a
subsequence $t_j\rightarrow\infty$ we have $\vn{\rk_{\rs^3}}_2^2\rightarrow0$.
Let us take a further subsequence of $\{t_j\}$ (also called $\{t_j\}$) along which $\vn{\rk_{\rs^3}}_2^2\searrow0$.
Note that
\begin{equation}
\label{EQrkfrombelow1}
\rk(u,t) = \int_{u_0}^u \rk_{\rs}\,d\rs + \rk(u_0,t)
         \ge -\int_{\rgamma} |\rk_{\rs}|\,d\rs + \delta_0
\,.
\end{equation}
The first term can be made small (keeping in mind the uniform bound on length from above and below) by taking $t=t_j$ for $j$ sufficiently large.
In particular, we can choose $j_0$ large enough so that
\[
	-\int_{\rgamma} |\rk_{\rs}|\,ds\bigg|_{t=t_{j_0}}
	\ge -\rL^\frac12 \bigg( \int_{\rgamma} |\rk_{\rs}|^2\,ds \bigg)^\frac12\bigg|_{t=t_{j_0}}
	\ge -\rL^\frac12 \bigg( \frac{\rL^4}{16\omega^4\pi^4} \int_{\rgamma} |\rk_{\rs^3}|^2\,ds \bigg)^\frac12\bigg|_{t=t_{j_0}}
	\ge -\frac{\delta_0}{2}
\,.
\]
Combining this estimate with \eqref{EQrkfrombelow1} yields
\[
\rk(u,t_j) \ge \frac{\delta_0}{2}
\,.
\]
The monotone decay of $\vn{\rk_{\rs^3}}_2^2$ along the sequence $\{t_j\}$ implies that the argument here works at any time $t_j$ where $j\ge j_0$.
This finishes the proof.
\end{proof}

The evolution equations for curvature and its derivatives change in the rescaled variables according to their homogeneity.
More precisely, consider a smooth functional $\SF[\gamma]:[0,T)\rightarrow\infty$ with degree $\lambda$.
The evolution of the functional in rescaled time on $\rgamma$ is
\[
	\frac{d}{dt}\SF[\rgamma](t) = \lambda \SF[\rgamma](t) + \SF[\rgamma]'(t)\,,\quad \SF[\rgamma]:[0,\infty)\rightarrow\infty
	\,.
\]
For instance, we have
\begin{align*}
	\rL' &= \rL - \int_{\rgamma} \rk^4 + \rk_{\rs}^2 \,d\rs
\\
	\rA' &= 2\rA - \int_{\rgamma} \rk^3\,d\rs\,,
\end{align*}
and so on.
Note in particular that a scale-invariant quantity (homogeneous of degree $\lambda = 0$) has `the same' evolution in the rescaled variables.

This implies the following uniform estimates.
The proof amounts only to the observation that the interpolation method of Dziuk-Kuwert-Sch\"atzle applies
also to the rescaled flow, despite there being a `zero order term' in the speed
of the flow.

\begin{thm}
\label{TMrallest}
Suppose $\gamma:I\times[0,T)\rightarrow\R^2$ is a compact Chen flow with
initial data $\gamma_0:I\rightarrow\R^2$ being an immersed $\omega$-circle
satisfying \eqref{EQhypothm}.

Let $\rgamma:I\times[0,\infty)\rightarrow\R^2$ be the rescaling of $\gamma$
around the final point $\SO$ (from Theorem \ref{TMfinalpoint}).
For each $m\in\N_0$ there exist constants $C_m$ depending only on $\omega$ and $\vn{\rk_{\rs^m}}_2^2(0)$ such that
\[
	\int_{\rgamma} \rk_{\rs^m}^2\,d\rs \le C_m
	\,,\qquad \text{ for all $t\in[0,\infty)$}\,.
\]
\end{thm}
\begin{proof}
We have the following version of \cite[equation (3.2)]{DKS2002}:
\begin{align*}
	\frac{d}{dt}\int_{\rgamma}\rk_{\rs^m}^2\,d\rs
	 + \int_{\rgamma}\rk_{\rs^{m+2}}^2\,d\rs
	\le
	  -(1+2m)\int_{\rgamma}\rk_{\rs^m}^2\,d\rs
	+ c_m\bigg(\int_{\rgamma}\rk^2\,d\rs\bigg)^{2m+5}
\,.
\end{align*}
Integration by parts yields the estimate
\[
	  -(1+2m)\int_{\rgamma}\rk_{\rs^m}^2\,d\rs
	\le \frac12\int_{\rgamma}\rk_{\rs^{m+2}}^2\,d\rs
	 + c_{m,1}\int_{\rgamma}\rk^2\,d\rs
\,,
\]
which gives, after absorption,
\begin{align*}
	\frac{d}{dt}\int_{\rgamma}\rk_{\rs^m}^2\,d\rs
	 + \frac12\int_{\rgamma}\rk_{\rs^{m+2}}^2\,d\rs
	\le
	  c_{m,2}\int_{\rgamma}\rk^2\,d\rs\bigg(1 + \Big(\int_{\rgamma}\rk^2\,d\rs\Big)^{2m+4}\bigg)
\,.
\end{align*}
Note that we have
\[
\int_{\rgamma}\rk_{\rs^{m}}^2\,d\rs
\le
	\frac{L^4}{(2\omega\pi)^4}
	\int_{\rgamma}\rk_{\rs^{m+2}}^2\,d\rs
\]
so that, with the uniform length estimates (Lemma \ref{LMrlb}), we have
\begin{equation}
\label{EQmainest}
	\frac{d}{dt}\int_{\rgamma}\rk_{\rs^m}^2\,d\rs
	\le
	   c_{m,2}\int_{\rgamma}\rk^2\,d\rs\bigg(1 + \Big(\int_{\rgamma}\rk^2\,d\rs\Big)^{2m+4}\bigg)
	 - D\int_{\rgamma}\rk_{\rs^{m}}^2\,d\rs
\,.
\end{equation}
Now we also have, from the uniform estimate \eqref{EQsql1} on $\SQ$, and the uniform estimate on length, that
\begin{equation}
\label{EQrcurvatureest}
	\int_{\rgamma}\rk^2\,d\rs \le \sqrt{\rL\vn{\rk}_4^4} \le \sqrt{\rL^{-2}\rSQ(0)} \le C_0
	\,.
\end{equation}
We therefore conclude from \eqref{EQmainest} that
\begin{equation}
\label{EQmainest1}
	\frac{d}{dt}\int_{\rgamma}\rk_{\rs^m}^2\,d\rs
	\le
	   c_{m,3}
	 - D\int_{\rgamma}\rk_{\rs^{m}}^2\,d\rs
\,.
\end{equation}
This implies that $\int_{\rgamma}\rk_{\rs^m}^2\,d\rs \le \frac{c_{m,3}}{D} =: C_m$, as required.
\end{proof}

The missing piece of information that we need to conclude uniform boundedness
of $\gamma(\cdot, t)$ in $C^m(I)$ for any $m$ is an estimate on $|\rgamma|$,
the position vector.
We establish this now.

\begin{prop}
\label{PNrgbdd}
Suppose $\gamma:I\times[0,T)\rightarrow\R^2$ is a compact Chen flow with
initial data $\gamma_0:I\rightarrow\R^2$ being an immersed $\omega$-circle
satisfying \eqref{EQhypothm}.

Let $\rgamma:I\times[0,\infty)\rightarrow\R^2$ be the rescaling of $\gamma$
around the final point $\SO$ (from Theorem \ref{TMfinalpoint}).
Then
\[
|\rgamma|(u,t) \le C
\]
for all $(u,t)\in I\times[0,\infty)$, where $C$ depends only on $\omega$, $\vn{k}_2^2(0)$, and $\rSQ(0)$.
\end{prop}
\begin{proof}
In this proof we use $C$ to denote a universal constant (one depending only on $\omega$, $\vn{\rk_{\rs^2}}_2^2(0)$, and $\rSQ(0)$) that may change from line to line.

From the definition of $\rgamma$ we see that the desired estimate is equivalent to
\begin{equation}
\label{EQdesired0}
	|\gamma(u, T-Te^{-4t}) - \SO| \le C(4T)^\frac14\, e^{-t}
\,.
\end{equation}
The blowup estimate (Theorem \ref{TMblowup}) implies that
\[
	T \ge c_B\bigg(\int_\gamma k^2\,ds\bigg)^{-4}\bigg|_{t=0}\,,
\]
so that \eqref{EQdesired0} is equivalent to
\begin{equation}
\label{EQdesired1}
	|\gamma(u, T-Te^{-4t}) - \SO| \le C e^{-t}
\,.
\end{equation}
We calculate (using Theorem \ref{TMfinalpoint})
\begin{equation}
\label{EQhaveest1}
\begin{split}
	|\gamma(u, T-Te^{-4t}) - \SO|
	&= |\gamma(u, T-Te^{-4t}) - \gamma(u,T)|
	\\
	&= \bigg|
		\int_{T-Te^{-4t}}^T \gamma_t(u,r)\,dr
		\bigg|
	\\
	&= \bigg|
		\int_{T-Te^{-4t}}^T
			(-k_{ss} + k^3)\nu + 3kk_s\tau
			\,dr
		\bigg|
	\\
	&\le
		\int_{T-Te^{-4t}}^T
			|k_{ss}|
			\,dr
		+ \int_{T-Te^{-4t}}^T
			|k|^3
			\,dr
		+ 3\int_{T-Te^{-4t}}^T
			|k|\,|k_s|
			\,dr
\,.
\end{split}
\end{equation}
We estimate each of the terms on the right hand side of \eqref{EQhaveest1} in turn.
First
\begin{equation}
\label{EQhaveest2}
\begin{split}
		\int_{T-Te^{-4t}}^T
			|k_{ss}|
			\,dr
	&\le
		\int_{T-Te^{-4t}}^T
			\vn{k_{s^3}}_1
			\,dr
	\\
	&\le
		\int_{T-Te^{-4t}}^T
			L^\frac12\vn{k_{s^3}}_2
			\,dr
	\\
	&\le
		\int_{T-Te^{-4t}}^T
			L^{-1}L^\frac32\vn{k_{s^3}}_2
			\,dr
	\\
	&\le
		\bigg(
		\int_{T-Te^{-4t}}^T
			L^{-2}
			\,dr
		\bigg)^\frac12
		\bigg(
		\int_{T-Te^{-4t}}^T
			L^3\vn{k_{s^3}}_2^2
			\,dr
		\bigg)^\frac12
	\,.
\end{split}
\end{equation}
Now from estimate \eqref{EQsql1} we find
\begin{equation}
\label{EQhaveest3}
		\int_{T-Te^{-4t}}^T
			L^3\vn{k_{s^3}}_2^2
			\,dr
	\le \varepsilon_2^{-1}\SQ(T-Te^{-4t})
	\le \varepsilon_2^{-1}\SQ(0)
\,.
\end{equation}
Combining \eqref{EQhaveest3} above with \eqref{EQhaveest2} before we obtain
\begin{equation}
\label{EQhaveest4}
		\int_{T-Te^{-4t}}^T
			|k_{ss}|
			\,dr
	\le
		C\bigg(
		\int_{T-Te^{-4t}}^T
			L^{-2}
			\,dr
		\bigg)^\frac12
\,.
\end{equation}
Now since the rescaled length satisfies $\rL(t) = e^tL(T-Te^{-4t})$, Lemma \ref{LMrlb} implies
\[
	ce^{-t} \le L(T-Te^{-4t}) \le Ce^{-t}
\,.
\]
In terms of the original variables, this equates to the estimate
\[
	c\Big(1-\frac{t}{T}\Big)^\frac14
	\le L(t) \le
	C\Big(1-\frac{t}{T}\Big)^\frac14
\,.
\]
Integration yields
\begin{equation}
\label{EQhaveest5}
	\int_{T-Te^{-4t}}^T
		L^{-2}
		\,dr
	\le
	c\int_{T-Te^{-4t}}^T
		\Big(1-\frac{r}{T}\Big)^{-\frac12}
		\,dr
	\le
	ce^{-2t}
\,.
\end{equation}
Combining \eqref{EQhaveest4} with \eqref{EQhaveest5} we see that
\begin{equation}
\label{EQterm1}
		\int_{T-Te^{-4t}}^T
			|k_{ss}|
			\,dr
	\le Ce^{-t}
\,.
\end{equation}
This deals with the first term on the right hand side of \eqref{EQhaveest1}.

For the second term, we first expand
\begin{equation}
\label{EQhaveest6}
\begin{split}
	\int_{T-Te^{-4t}}^T |k|^3\,dr
	&\le
		\int_{T-Te^{-4t}}^T
		(k-\k)^3 + 3\k(k-\k)^2 + 3\k^2(k-\k) + \k^3
		\,dr
\end{split}
\end{equation}
We estimate each term in turn as follows:
\begin{align*}
		\int_{T-Te^{-4t}}^T
		\k^3
		\,dr
	&\le
		8\omega^3\pi^3
		\int_{T-Te^{-4t}}^T L^{-3} \,dr
	\\
	&\le
		C
		\int_{T-Te^{-4t}}^T
		\Big(1-\frac{r}{T}\Big)^{-\frac34}
		\,dr
	\\
	&\le
		Ce^{-t}
\\
		3\int_{T-Te^{-4t}}^T
		\k^2(k-\k)
		\,dr
	&\le
		12\omega^2\pi^2\int_{T-Te^{-4t}}^T
		L^{-2}L^\frac12\vn{k_s}_2
		\,dr
	\\
	&\le
		C\int_{T-Te^{-4t}}^T
		L^{-3}\sqrt{L^3\vn{k_s}_2^2}
		\,dr
	\\
	&\le
		C\int_{T-Te^{-4t}}^T
		L^{-3}
		\,dr
	\\
	&\le
		Ce^{-t}
\\
		3\int_{T-Te^{-4t}}^T
		\k(k-\k)^2
		\,dr
	&\le
		C\int_{T-Te^{-4t}}^T
		L^{-3}L^3\vn{k_s}_2^2
		\,dr
	\\
	&\le
		C\int_{T-Te^{-4t}}^T
		L^{-3}
		\,dr
	\\
	&\le
		Ce^{-t}
\\
		\int_{T-Te^{-4t}}^T
		(k-\k)^3
		\,dr
	&\le
		C\int_{T-Te^{-4t}}^T
		\vn{k_s}_1^3
		\,dr
	\\
	&\le
		C\int_{T-Te^{-4t}}^T
		L^{\frac32}\vn{k_s}_2^3
		\,dr
	\\
	&\le
		C\int_{T-Te^{-4t}}^T
		L^{-3}\Big(L^3\vn{k_s}_2^2\Big)^\frac32
		\,dr
	\\
	&\le
		Ce^{-t}
\,.
\end{align*}
Therefore combining all of these estimates with \eqref{EQhaveest6} we find
\begin{equation}
\label{EQhaveest7}
\begin{split}
	\int_{T-Te^{-4t}}^T |k|^3\,dr
	&\le
		Ce^{-t}
\,.
\end{split}
\end{equation}
For the remaining term on the right hand side of \eqref{EQhaveest1} we similarly expand
\begin{equation}
\label{EQhaveest8}
\begin{split}
	3\int_{T-Te^{-4t}}^T |k|\,|k_s|\,dr
	&\le
	3\int_{T-Te^{-4t}}^T |k-\k|\,|k_s|\,dr
	+ 3\int_{T-Te^{-4t}}^T \k\,|k_s|\,dr
\,.
\end{split}
\end{equation}
We estimate each of these terms by
\begin{align*}
	3\int_{T-Te^{-4t}}^T |k-\k|\,|k_s|\,dr
	&\le
	3\int_{T-Te^{-4t}}^T |k_s|L^{-1}\sqrt{L^3\vn{k_s}_2^2}\,dr
	\\
	&\le
	C\int_{T-Te^{-4t}}^T \vn{k_{s^3}}_1L^{-1}\,dr
	\\
	&\le
	C\int_{T-Te^{-4t}}^T L^\frac12\vn{k_{s^3}}_2L^{-1}\,dr
	\\
	&\le
	C\bigg(
	\int_{T-Te^{-4t}}^T L^3\vn{k_{s^3}}_2^2\,dr
	\bigg)^\frac12
	 \bigg(
	 \int_{T-Te^{-4t}}^T L^{-2}\,dr
	\bigg)^\frac12
	\\
	&\le
		Ce^{-t}
\\
	3\int_{T-Te^{-4t}}^T \k\,|k_s|\,dr
	&\le
	C\int_{T-Te^{-4t}}^T \vn{k_{s^3}}_1L^{-1}\,dr
	\\
	&\le
		Ce^{-t}
\,.
\end{align*}
Now combining these with \eqref{EQhaveest8} gives
\begin{equation}
\label{EQhaveest9}
\begin{split}
	3\int_{T-Te^{-4t}}^T |k|\,|k_s|\,dr
	&\le
	Ce^{-t}
\,.
\end{split}
\end{equation}
Finally combining each of \eqref{EQterm1}, \eqref{EQhaveest7}, and \eqref{EQhaveest9} with \eqref{EQhaveest1} yields \eqref{EQdesired1}, and so finishes the proof.
\end{proof}

Therefore $\rgamma$ and all of its derivatives are uniformly bounded.
In order to conclude convergence, we obtain decay of the isoperimetric ratio
\[
	\rSI = 1 - \frac{4\omega\pi\rA}{\rL^2}
	\,.
\]

\begin{prop}
\label{PNdecayrsi}
Suppose $\gamma:I\times[0,T)\rightarrow\R^2$ is a compact Chen flow with
initial data $\gamma_0:I\rightarrow\R^2$ being an immersed $\omega$-circle
satisfying \eqref{EQhypothm}.

Let $\rgamma:I\times[0,\infty)\rightarrow\R^2$ be the rescaling of $\gamma$
around the final point $\SO$ (from Theorem \ref{TMfinalpoint}).
Then there exists constants $C < \infty$, $\varepsilon_3 > 0$ depending only on $\omega$, $\rSQ(0)$ and $\vn{\rk_{\rs^2}}_2^2(0)$ such that
\[
\rSI(t) \le Ce^{-\varepsilon_3 t}\,.
\]
\end{prop}
\begin{proof}
First let us calculate
\begin{equation}
\label{EQdecay0}
\begin{split}
	\SI'
	&= -\frac{4\omega\pi}{\rL^2}\rA' + 2\frac{4\omega\pi}{\rL^3}\rL'
\\
	&= \frac{4\omega\pi}{\rL^2}\bigg(
		\int_\rgamma \rk^3\,d\rs
		- 2\frac{\rA}{\rL}\int_\rgamma \rk^4 + \rk_{\rs}^2\,d\rs
	\bigg)
\,.
\end{split}
\end{equation}
Expanding out $\rk^3$ and $\rk^4$ in terms of $(\rk - \ork)$ and $\ork$, we find
\begin{equation}
\label{EQdecay1}
\begin{split}
		\int_\rgamma \rk^3\,d\rs
		- 2\frac{\rA}{\rL}\int_\rgamma \rk^4 \,d\rs
	&=
		- 2\frac{\rA}{\rL}\int_\rgamma (\rk-\ork)^4 \,d\rs
		+ \Big(1 - 8\ork\frac{\rA}{\rL}\Big) \int_\rgamma (\rk-\ork)^3 \,d\rs
\\&\qquad
		+ \ork\Big(3 - 12\ork\frac{\rA}{\rL}\Big) \int_\rgamma (\rk-\ork)^2 \,d\rs
		+ \ork^3\Big(1 - 2\ork\frac{\rA}{\rL}\Big) \rL
\,.
\end{split}
\end{equation}
Since $1 - 2\ork\rA/\rL = \rSI$, the above is in fact
\begin{equation}
\label{EQdecay2}
\begin{split}
		\int_\rgamma \rk^3\,d\rs
		- 2\frac{\rA}{\rL}\int_\rgamma \rk^4 \,d\rs
	&=
		- 2\frac{\rA}{\rL}\int_\rgamma (\rk-\ork)^4 \,d\rs
		+ \Big(\rSI - \frac{12\omega\pi\rA}{\rL^2}\Big) \int_\rgamma (\rk-\ork)^3 \,d\rs
\\&\qquad
		+ 3\ork\Big(\rSI - \frac{4\omega\pi\rA}{\rL^2}\Big) \int_\rgamma (\rk-\ork)^2 \,d\rs
		+ \ork^3\rL\rSI
\,.
\end{split}
\end{equation}
Now we use the estimate
\begin{equation}
\label{EQdecayest1}
	\rSI \le \frac{\rL}{8\omega^2\pi^2} \int_{\rgamma} (\rk-\ork)^2\,d\rs
\end{equation}
(which follows by using the Fourier series technique from \cite[proof of
Theorem 6.1]{Ideal}, a proof is included in the appendix of \cite{Okinawa}), the isoperimetric inequality, and the simpler estimate
\begin{equation}
\label{EQdecayest2}
	\vn{\rk-\ork}_\infty^2 \le \frac{\rL}{2\omega\pi} \int_{\rgamma} \rk_{\rs}^2\,d\rs
\end{equation}
with \eqref{EQdecay2} to conclude
\begin{equation}
\label{EQdecay3}
\begin{split}
		\int_\rgamma \rk^3\,d\rs
		- 2\frac{\rA}{\rL}\int_\rgamma \rk^4 \,d\rs
	&\le
		- 2\frac{\rA}{\rL}\int_\rgamma (\rk-\ork)^4 \,d\rs
		- 3\ork\frac{4\omega\pi\rA}{\rL^2} \int_\rgamma (\rk-\ork)^2 \,d\rs
\\&\qquad
		+ \Big(\frac{\rL^3}{2(2\omega\pi)^4}\vn{\rk_{\rs}}_2^2 + 3\Big)\frac{\rL^\frac12}{(2\omega\pi)^\frac12} \vn{\rk_{\rs}}_2 \int_\rgamma (\rk-\ork)^2 \,d\rs
\\&\qquad
		+ 3\ork\frac{\rL^3}{2(2\omega\pi)^4}\vn{\rk_{\rs}}_2^2
			\int_\rgamma (\rk-\ork)^2 \,d\rs
		+ \ork^3\frac{\rL^2}{8\omega^2\pi^2} \int_{\rgamma} (\rk-\ork)^2\,d\rs
\,.
\end{split}
\end{equation}
Since
\[
		- 3\ork\frac{4\omega\pi\rA}{\rL^2}
		+ \ork^3\frac{\rL^2}{8\omega^2\pi^2}
	\le \ork\Big(
		  \frac12
		- \frac12\frac{4\omega\pi\rA}{\rL^2}
		- \frac52\frac{4\omega\pi\rA}{\rL^2}
		\Big)
	  = \ork\Big(
		  \rSI
		- \frac52\frac{4\omega\pi\rA}{\rL^2}
		\Big)
\]
we may refine \eqref{EQdecay3} to
\begin{equation}
\label{EQdecay4}
\begin{split}
		\int_\rgamma \rk^3\,d\rs
		- 2\frac{\rA}{\rL}\int_\rgamma \rk^4 \,d\rs
	&\le
		- 2\frac{\rA}{\rL}\int_\rgamma (\rk-\ork)^4 \,d\rs
		- \frac{20\omega^2\pi^2\rA}{\rL^3}
			\int_\rgamma (\rk-\ork)^2 \,d\rs
\\&\qquad
		+ \Big(\frac{\rL^3}{2(2\omega\pi)^4}\vn{\rk_{\rs}}_2^2 + 3\Big)\frac{\rL^\frac12}{(2\omega\pi)^\frac12} \vn{\rk_{\rs}}_2 \int_\rgamma (\rk-\ork)^2 \,d\rs
\\&\qquad
		+ 3\ork\frac{\rL^3}{2(2\omega\pi)^4}\vn{\rk_{\rs}}_2^2
			\int_\rgamma (\rk-\ork)^2 \,d\rs
	\\
	&\le
		- 2\frac{\rA}{\rL}\int_\rgamma (\rk-\ork)^4 \,d\rs
		+ \bigg(
		  \Big(\frac{\rL^3}{2(2\omega\pi)^4}\vn{\rk_{\rs}}_2^2 + 3\Big)\frac{\rL^\frac12}{(2\omega\pi)^\frac12} \vn{\rk_{\rs}}_2
\\&\qquad
		+ 3\ork\frac{\rL^3}{2(2\omega\pi)^4}\vn{\rk_{\rs}}_2^2
		- \frac{20\omega^2\pi^2\rA}{\rL^3} \bigg)
			\int_\rgamma (\rk-\ork)^2 \,d\rs
\,.
\end{split}
\end{equation}
The uniform length bound (Lemma \ref{LMrlb}) and Lemma \ref{LMkpos} implies that there exists a sequence $\{t_j\}$ along which $\vn{\rk_{\rs}}_2(t_j)\rightarrow0$.
Furthermore, the sequence $\{t_j\}$ has the property that $\rk(u,t_j) \ge \frac{\delta_0}{2}$, where $\delta_0 = \delta_0(\rSQ(0))$.
Therefore the curves $\rgamma(\cdot, t_j)$ are star-shaped with respect to some point $p_j\in\R^2$ and so $\rA(t_j) > 0$; in particular we have the estimate
\[
	A(t_j) \ge \frac{\rL(t_j)}{2\vn{\rk}_\infty}
\]
which follows because
\[
	\rL = -\int_\rgamma \rk\IP{\rgamma}{\rnu}\,d\rs
	  = -\int_\rgamma \rk\IP{\rgamma-p_j}{\rnu}\,d\rs
	  \le 2\vn{\rk}_\infty \rA(t_j)
	\,.
\]
Since $\vn{\rk_{\rs}}_2(t_j)\rightarrow0$ we have (using also again the uniform length estimates)
\[
\vn{\rk}_\infty \le \rL^\frac12(t_j)\vn{\rk_{\rs}}_2(t_j) + \frac{2\omega\pi}{\rL(t_j)}
	\le C_k
\]
where $C_k$ is a constant depending only on $\rSQ(0)$.
Therefore we find
\begin{equation}
\label{EQuniformrareaest}
	A(t_j) \ge C_A
\end{equation}
where $C_A$ is a constant depending onlty on $\rSQ(0)$.

Combining \eqref{EQuniformrareaest} with the uniform length estimates and
\eqref{EQdecay4}, as well as taking $j$ sufficiently large, we find a
$\varepsilon_3 > 0$ depending only on $\rSQ(0)$ such that
\begin{equation}
\label{EQdecay5}
\begin{split}
		\int_\rgamma \rk^3\,d\rs
		- 2\frac{\rA}{\rL}\int_\rgamma \rk^4 \,d\rs
	&\le
		- 2\frac{\rA}{\rL}\int_\rgamma (\rk-\ork)^4 \,d\rs
		- \varepsilon_3\frac{\rL^3}{32\omega^3\pi^3}
			\int_\rgamma (\rk-\ork)^2 \,d\rs\,,\quad\text{for $t=t_j$}
\,.
\end{split}
\end{equation}
Using \eqref{EQdecay5} with \eqref{EQdecay0}, as well as \eqref{EQdecayest1}, we find a $\delta, \varepsilon_3 > 0$ such that
\begin{equation}
\label{EQdecayest3}
	\vn{\rk_{\rs}}_2(t) < \delta \quad\Longrightarrow\quad \rSI'(t) \le -\varepsilon_3\rSI(t)
	\,.
\end{equation}
and $\vn{\rk_{\rs}}_2(t_j) < \delta$.
This estimate implies that $\rSI(t_j)$ is instantaneously decreasing.
The estimates from the appendix of \cite{Okinawa} yield (as well
as Theorem \ref{TMrallest}, the estimate for $\vn{k}_\infty$, and the uniform
length estimate)
\[
	\vn{\rk_{\rs}}_2^4 \le \vn{\rk_{\rs^2}}_2^2 \vn{\rk-\ork}_2^2
		\le C\SI^\frac12\Big(1 + \vn{\rk_{\rs}}_2^2\Big)^\frac12
		\le C\SI^\frac12\,.
\]
Therefore
\[
	\vn{\rk_{\rs}}_2^2(t) \le C\SI^\frac14(t)\,.
\]
Now the estimate $\SI \le \frac{\rL^3}{32\omega^4\pi^4}\vn{\rk_{\rs}}_2^2$
implies that $\SI(t_j)\rightarrow0$ along the sequence $\{t_j\}$.
In particular there exists a $J$ such that $\SI(t_J) \le \frac{\delta^8}{C^4}$.
Then \eqref{EQdecayest3} implies that $\SI'(t_J) \le 0$ and so the hypothesis
that $\vn{\rk_{\rs}} < \delta$ is preserved.
This means that the estimate
\[
	\rSI'(t) \le -\varepsilon_3\rSI(t)
\]
holds for all $t\ge t_J$.
Therefore
\[
	\rSI(t) \le \rSI(t_J)e^{-\varepsilon_3t}
\]
for $t\ge t_J$; and so
\[
	\rSI(t) \le Ce^{-\varepsilon_3t}
\]
for all $t\ge 0$, as required.
\end{proof}

Interpolation now yields exponential decay of $\vn{\rk-\ork}_2$ as well
as all derivatives of curvature.

\begin{cor}
\label{CYrcurvdecay}
Suppose $\gamma:I\times[0,T)\rightarrow\R^2$ is a compact Chen flow with
initial data $\gamma_0:I\rightarrow\R^2$ being an immersed $\omega$-circle
satisfying \eqref{EQhypothm}.

Let $\rgamma:I\times[0,\infty)\rightarrow\R^2$ be the rescaling of $\gamma$
around the final point $\SO$ (from Theorem \ref{TMfinalpoint}).
Then there exists constants $C_{m,1} < \infty$ depending only on $\omega$, $\rSQ(0)$ and $\vn{\rk_{\rs^l}}_2^2(0)$ for $l\in\{0,\ldots,2m\}$ such that
\begin{equation}
\label{EQdecayofrosc}
	\rL\int_{\rgamma} (\rk - \ork)^2\,d\rs \le C_{0,1}e^{-\frac{\varepsilon_3}{2}t}
\end{equation}
and
\begin{equation}
\label{EQdecayofrall}
	\rL^{2m+1}\int_{\rgamma} \rk_{\rs^m}^2\,d\rs \le C_{m,1}e^{-\frac{\varepsilon_3}{4}t}
	\,,\qquad \text{ for all $t\in[0,\infty)$}\,.
\end{equation}
\end{cor}
\begin{proof}
All of these decay estimates follow from Proposition \ref{PNdecayrsi} and interpolation.
The basic interpolation estimate is
\begin{equation}
\label{EQinterpiso}
	\bigg(\rL\int_{\rgamma} (\rk-\ork)^2\,d\rs\bigg)^2
	\le
		\rL^3\rSI
		\int_{\rgamma} \frac{2\omega\pi}{\rL}\rk^3 + \rk_{\rs}^2\,d\rs
\,.
\end{equation}
This estimate (also used in the proof of Proposition \ref{PNdecayrsi} above)
follows using a Fourier series method that we learned from
Ben Andrews; it is exposed in the appendix of \cite{Okinawa}.
Using our uniform estimates for length and curvature, \eqref{EQinterpiso}
implies the decay estimate \eqref{EQdecayofrosc}.

The decay of higher derivatives now follows simply by using integration by
parts together with the H\"older inequality, and then an application of
\eqref{EQdecayofrosc} (and our uniform estimates):
\begin{align*}
\vn{\rk_{\rs^m}}_2^2
	&= \vn{(\rk-\ork)_{\rs^m}}_2^2
\\
	&\le \vn{\rk_{\rs^(2m)}}_2 \vn{\rk-\ork}_2
\\
	&\le Ce^{-\frac{\varepsilon_3}{4}t}
\,.
\end{align*}
This finishes the proof.
\end{proof}

Combining what we have proven so far, the rescaled flow converges exponentially fast along each subsequence to an (for now, possibly different depending on the subsequence) $\omega$-circle.
We finally identify to which $\omega$-circle the rescaled flow is converging in the following result.

\begin{thm}
Suppose $\gamma:I\times[0,T)\rightarrow\R^2$ is a compact Chen flow with
initial data $\gamma_0:I\rightarrow\R^2$ being an immersed $\omega$-circle
satisfying \eqref{EQhypothm}.

Let $\rgamma:I\times[0,\infty)\rightarrow\R^2$ be the rescaling of $\gamma$
around the final point $\SO$ (from Theorem \ref{TMfinalpoint}).

Then $\rgamma(I,t)$ converges to the unit $\omega$-circle centred at the
origin, with the curvature $\rk$ converging exponentially fast toward its average and
all derivatives of curvature covnerging exponentially fast to zero.
\end{thm}
\begin{proof}
The flow is converging to a circle, so $\IP{\rgamma_t}{\rnu} \rightarrow 0$, which implies
\[
	(\hat k_{\hat s\hat s} - \hat k^3 - \IP{\rgamma}{\rnu}) \rightarrow 0
	\,.
\]
However, since we know that the flow is also becoming rounder exponentially fast, this means that $\rk^3 \rightarrow -\IP{\rgamma}{\rnu} \rightarrow 0$ (also exponentially fast).
If we take a subsequence of times converging to a circle with centre $p\in\R^2$ and of radius $R$, we see that $p$ and $R$ must be the origin and one respectively.
That is, only one circle is possible as the limit of the flow.
This finishes the proof.
\end{proof}

\section{Numerics for Chen's flow of curves}
\label{sect:Numerics}

\subsection{The algorithm}\label{sect:AlgIntro}
Here, we describe a numerical algorithm for Chen's flow~\eqref{ChenFlow} for a family of immersed curves $\gamma : \R/\Z \times I \rightarrow \R^d$, where $I \subset \R$ an interval.
Note that earlier we considered the domain of our compact curves to be $\S$, but for the numerics we consider it to be $\R/\Z$.
Of course, it is trivial to change between these two conventions, but for the numerics $\R/\Z$ is more convenient.

We will first write a semi-discretisation of our flow by discretising in space.
Following this we will discretise the resulting system of ODE\@.
This approach is commonly known as the method of lines.
We note that our approach does not require the curve to be planar.
Our discussion will be formal and heuristic, being completely rigorous would take us well-beyond our current scope.

Our algorithm is parametric, and hence we will approximate $\gamma : \R/\Z \rightarrow \R^d$ by $\gamma_i = \gamma(u_i)$, where $u_i = i/N$, $0 \leq i \leq N - 1$, and $N \in \N_{\geq 3}$.
We set $\Delta u = 1/N$, and note that one may think of $i$ as an element of $\Z/N\Z$.

We will not discretise~\eqref{ChenFlow} directly.
The system we begin with is
\begin{equation}\label{eq:ModChenFlow}
\left\{
\begin{aligned}
\p_\sigma \gamma &= \left(-\frac{a}{1 +  \|\kappa\|_{L^\infty(\R/\Z)}^4}(1 - P(\gamma))\p_s^4 + P(\gamma)\p_u^2\right) \gamma =: S(\gamma) \gamma,
\\
\p_\sigma t &= \frac{a}{1 +  \|\kappa\|_{L^\infty(\R/\Z)}^4},
\end{aligned}
\right.
\end{equation}
where $P(\gamma)X = \langle X, \tau \rangle \tau$ and $a > 0$.
Note that $S(\gamma)$ is linear.

There are two differences between~\eqref{ChenFlow} and~\eqref{eq:ModChenFlow}.
The first is that we have modified the tangential component of the flow.
We only consider the component of $\gamma_{s^4}$ normal to $\gamma$, and we add a tangential velocity which acts to reparaterise the curve to constant speed.
Taken on its own this latter velocity is the harmonic map heat flow of maps from $\R/\Z$ into $\gamma$ at a fixed time.
This coupling between Chen's flow and the harmonic map heat flow to reparameterise the curve is DeTurk's trick.
See~\cite{ElliottFritz2017} for further interesting connections between DeTurk's trick and the numerical analysis of geometric flows.
In the context of our scheme, this tangential velocity is important, because it keeps the $\gamma_i$ close to equidistributed, and hence stops them from coalescing.

In~\cite{BGN2011} Barrett, Garcke, and N\"{u}rnberg present a scheme for curve shortening flow and curve diffusion flow in which the points of the discretisation of $\gamma$ remain equidistributed.
However, a drawback of their scheme is its fully implicit nature which requires the solution to highly non-linear algebraic equations at each time step.

The second difference is that we have introduced an auxiliary time variable $\sigma$.
This is an instance of the Sundman transform, and it has various desirable effects.
For instance, if $\gamma$ is a homothetically shrinking solution of Chen's flow then
\begin{equation*}
\p_t L[\gamma] = -c (L[\gamma])^{-3},
\end{equation*}
where $c > 0$ is a constant depending only on $\gamma$.
This implies that $\gamma$ shrinks to a point in finite-time.
On the other hand, looking at $L[\gamma]$ using the auxiliary time scale, while ignoring $a$ for now, yields
\begin{equation}\label{eq:RescaledLengthODE}
\p_\sigma L[\gamma] = \frac{-c L[\gamma]}{(L[\gamma])^4 + d},
\end{equation}
where $c, d > 0$ are constants depending only on $\gamma$.
This removes the finite-time singularity.

Another, and more practical, reason for the introduction of the auxiliary time was to avoid a numerical artefact when numerically evolving the Lemniscate of Bernoulli which is a homothetically shrinking solution to Chen's flow, see Figure~\ref{fig:ShrinkingLemniscate}.
As the Lemniscate shrinks the natural time-scale of the flow quickens.
A precise consequence of this is that the time it takes the Lemniscate's length to half monotonically decreases to zero.
Consider an algorithm using a fixed-time step $\Delta t$, as $t$ approaches $T$ the increment we are taking in the natural time-scale is approaching infinity, and this large step distorts the Lemniscate so that the ratio between its vertical extent and its horizontal extent approaches zero as we approach the blowup time.
We can see from~\eqref{eq:RescaledLengthODE} that the amount of $\sigma$-time it takes for the length to half is approximately constant as we approach the blowup time.
When using constant $\Delta \sigma$ steps we no longer observe this distortion of the Lemniscate.
Another way one may view this, is that we are, roughly, using an adaptive step-size with $\Delta t \propto 1/(1 + \|\kappa\|_{L^\infty(\R/\Z)}^4)$.

The purpose of $a$ in~\eqref{eq:ModChenFlow} is to adjust the time-scale over which the reparameterisation occurs.
As $a \searrow 0$ we expect that the parameterisation to stay closer to constant-speed, and the $\gamma_i$ to stay more closely equidistributed.
We use
\begin{equation*}
\max_{0 \leq i \leq n - 1} |\gamma_{i + 1} -\gamma_i|/\min_{0 \leq i \leq n - 1} |\gamma_{i + 1} -\gamma_i|.
\end{equation*}
to quantify how far from equidistribution the $\gamma_i$ are.
After each time step we verify that this values is below a fixed threshold.
In this work we have chosen this threshold to be two.
If after any $\sigma$-time step the above quantity becomes higher than this threshold our program prints a warning, after which the simulation can be re-run with a smaller $a$.
We have found that our initial setting of $a = 1$ works fine in all our computations, except for the one illustrated in Figure~\ref{fig:Mayer} for which we chose $a = 0.1$.

Next, we discuss how we discretise~\eqref{eq:ModChenFlow} in space.
By examining~\eqref{eq:ModChenFlow} we see that the quantities we need to approximate are $\gamma_{uu}$, $\tau$, $\gamma_{s^4}$, and $\|\kappa\|_{L^\infty(\R/\Z)}$.
We use standard centred-finite-differences to approximate $\gamma_{uu}$ and $\tau$:
\begin{equation*}
\gamma_{uu}(u_i) \approx \frac{\gamma_{i + 1} + \gamma_{i - 1} - 2\gamma_i}{\Delta u^2}
\text{ and }
\tau(u_i) \approx \frac{\gamma_{i + 1} - \gamma_{i - 1}}{|\gamma_{i + 1} - \gamma_{i - 1}|}.
\end{equation*}
We approximate $\gamma_{s^4}$ by starting with a centred-finite-difference approximation to $\p_s^2$, denoted by $\delta_s^{(2)}$, and applying this twice to approximate $\gamma_{s^4}$ by $(\delta_s^{(2)})^2 \gamma$.
Suppose that $f : \R/\Z \rightarrow \R^k$, and set $f_i = f(u_i)$.
First we have
\begin{equation*}
\p_s f(u_{i + 1/2}) \approx \delta_s^{(1)} f_{i + 1/2} = \frac{f_{i + 1} - f_i}{|\gamma_{i + 1} - \gamma_i|},
\end{equation*}
where $u_{i + 1/2} = (u_i + u_{i + 1})/2$.
We use $\delta_s^{(1)}$ to obtain
\begin{equation*}
\delta_s^{(2)} f_i = \frac{\delta_s^{(1)} f_{i + \frac{1}{2}} - \delta_s^{(1)} f_{i - \frac{1}{2}}}{\frac{1}{2}(|\gamma_{i+1} - \gamma_i| + |\gamma_i - \gamma_{i - 1}|)}.
\end{equation*}
Finally, we approximate $\|\kappa\|_{L^\infty(\R/\Z)}$:
\begin{equation*}
\|\kappa\|_{L^\infty(\R/\Z)} \approx \max_{0 \leq i \leq n - 1} |\delta_s^{(2)} \gamma_i|.
\end{equation*}
By expanding $\gamma$ and $f$ in Taylor series one can verify that these approximation are of second-order in $\Delta u$.

Much of the relevant literature follows~\cite{Dziuk1994} where Dziuk uses finite-element calculations to obtain spatial discretisations.
However, our discretisations are, essentially, the same, and we believe that a finite-difference approach is conceptually simpler.
Substituting these approximations into~\eqref{eq:ModChenFlow} yields a system of ODE for the quantities $t, \gamma_0, \gamma_1, \dots, \gamma_{n - 1}$.

Next, we discretise this system of ODE to obtain our final full-discretisation.
We set
\begin{equation*}
  w(\sigma) = (x_0(\sigma), y_0(\sigma), x_1(\sigma), y_1(\sigma), \dots, x_{n - 1}(\sigma), y_{n - 1}(\sigma))^T.
\end{equation*}
With this we can write our system of ODE as
\begin{equation}\label{eq:ODESys}
\p_\sigma w = A(w)w
\text{ and }
\p_\sigma t = b(w),
\end{equation}
where $A(w)$ is a $2n \times 2n$ matrix.
In the case of Chen's flow $A(w)$ is a band matrix, except for some entries in the bottom-left and top-right corners which occur due to periodicity.

Next, we present our full-discretisation.
In this discussion we forget Chen's flow and simply work with~\eqref{eq:ODESys}, as this highlights the general nature of our approach, which could easily apply to other geometric flows, and quasi-linear PDE in general.
Suppose that we have a solution $(w, t)$ to~\eqref{eq:ODESys} for $\sigma \in [0, T]$.
Our scheme yields a sequence of $\sigma_j = j\Delta\sigma$, $t_j = t(\sigma_j)$ and $w_j = w(\sigma_j)$, where $\Delta \sigma = T/M$, $M \in \N$, and $j \in \{0, \dots, M\}$.
Given $t_j$ and $w_j$ we describe how we compute $t_{j + 1}$ and $w_{j + 1}$:
\begin{enumerate}[1.]
\item
Compute $A_j = A(w_j)$.
\item
Solve for the unique $k_1 \in \R^{2n}$ that satisfies $(I - \Delta\sigma A_j) k_1 = w_j$.
\item
Set $t_{j + 1} = t_j + \Delta\sigma (b(w_j) + b(k_1))/2$.
\item
Compute $\t{A}_{j + 1} = A(k_1)$.
\item
Solve for the unique $k_2 \in \C^{2n}$ that satisfies $(\Delta\sigma (A_j + \t{A}_{j+1})/2 - (i + 1) I) k_2 = i w_j$.
\item
Set $w_{j + 1} = -2 \Re[k_2]$.
\end{enumerate}
This algorithm relies on the $(0, 1)$- and $(0, 2)$-Pad\'e approximations of the exponential function, and the partial-fraction decomposition of the latter approximation as presented in~\cite{KTV1993}.
We use these particular Pad\'e approximations so that our scheme is more stable and immune from spurious oscillations.
Its essence is the solving linear systems which in the specific case of Chen's flow are highly sparse.
One may view this as a semi-implicit scheme and as an extension of the one used in~\cite{Dziuk1994} and~\cite{DKS2002}.
These works studied the curve-shortening flow, the gradient flow of the elastic energy, and the curve diffusion flow.
In fact, if we took $w_{j + 1} = k_1$ we arrive at their scheme.
Whereas the time discretisations presented in~\cite{BGN2011,Dziuk1994,DKS2002} are first-order, our scheme is second-order.
Therefore, in order to achieve a second order full-discretisation we only need $\Delta\sigma \propto \Delta u$, as opposed to $\Delta\sigma \propto \Delta u^2$, which gives our scheme better asymptotic complexity.
For more details regarding the derivation of the scheme refer to Appendix~\ref{app:ODEMethodDetails}.

\subsection{Experimental convergence}
In this section we empirically demonstrate that our method is of second order.
We do this by using a closed-form solution to~\eqref{eq:ModChenFlow}, which we compare to the output of our numerical scheme.

We know a circle shrinks homothetically to a point under the flow.
We make the ansatz
\begin{equation*}
\gamma(u, \sigma) = r(\sigma) (\cos(2\pi g(u, \sigma)), \sin(2\pi g(u, \sigma))),
\end{equation*}
where $g : \R/\Z \times I \rightarrow \R/\Z$ and $I \subset \R$ is an interval.
Substituting this ansatz into~\eqref{eq:ModChenFlow} with $a = 1$ yields
\begin{equation*}
\p_\sigma r = \frac{-r}{1 + r^4}, \, \p_\sigma t = \frac{r^4}{1 + r^4}, \text{ and } \p_\sigma g = g_{uu}.
\end{equation*}
We take $r(0) = 1$ and $t(0) = 0$.
For $g$ we take
\begin{equation*}
g(u, \sigma) = u + \frac{1}{50} e^{-16\pi^2\sigma} \sin(4\pi u).
\end{equation*}
Note that $g_u > 0$, and hence our parameterisation is non-degenerate.
Using our numerical scheme we integrate from $\sigma_0 = 0$ until $\sigma_f = (1 + 2\log 2)/8$.
We chose $\sigma_f$ so that $t(\sigma_f) = 1/8 = T/2$.
We set $\Delta\sigma = 4\sigma_f/(25N) \approx 0.048/n$, and chose $\Delta\sigma$ so that when $N = 16$ the scheme takes $100$ $\sigma$-steps.
We chose this number of steps as a balance between too many steps which makes our algorithm slower, and too few steps which means we need to make $N$ very large in order to observe numerical convergence in the experimental order of convergence.

Our scheme returns a sequence of $(\sigma_j, t_j, \gamma_j)$ values.
We take the initial data of $\sigma_0 = t_0 = 0$ and
\begin{equation*}
\gamma_{0;i} = (\cos(2\pi g(i/N, 0)), \sin(2\pi g(i/N, 0))) \text{ for } 0 \leq i \leq n - 1.
\end{equation*}
The $\gamma_{j; i}$ is our approximation of $\gamma(i/N, t_j)$ which we know must be a circle of radius $(1 - 4t_j)^{1/4}$ and $g(\cdot, \sigma_j)$ gives the parameterisation of this circle.
More precisely, for each iterate $(\sigma_j, t_j, \gamma_j)$ we compute
\begin{equation*}
\max_{0 \leq i \leq n - 1} |\gamma_{j; i} - (1 - 4t_j)^{1/4} (\cos(2\pi g(i/n, \sigma_j)), \sin(2\pi g(i/n, \sigma_j)))|.
\end{equation*}
We take the maximum of these values over the iterates and call the result the $L^\infty$-error.

In Table~\ref{table:Errors} we tabulate these values, and also list the experimental order of convergence (EOC) via the formula $\log_2(E_{N/2}) - \log_2(E_N)$, where $E_N$ is the $L^\infty$-error when using $N$ points to approximate $\gamma$.

\begin{table}[!ht]
\begin{center}
\begin{tabular}{lll}
$n$ & $L^\infty$-error & EOC\\
\toprule
$16$ & $4.6963$e$-3$ & --\\
$32$ & $1.4200$e$-3$ & $1.73$\\
$64$ & $3.9067$e$-4$ & $1.86$\\
$128$ & $1.0230$e$-4$ & $1.93$\\
$256$ & $2.6158$e$-5$ & $1.97$\\
$512$ & $6.6132$e$-6$ & $1.98$\\
$1024$ & $1.6625$e$-6$ & $1.99$\\
$2048$ & $4.1687$e$-7$ & $2.00$\\
\end{tabular}
\caption{Absolute errors and experimental orders of convergence for the test problem.}\label{table:Errors}
\end{center}
\end{table}

\appendix

\section{Discretisation of the ODE system}\label{app:ODEMethodDetails}
Our time-discretisation of ODE of the form~\eqref{eq:ODESys} is second-order, that is, it has a local-truncation error of $\bigO(\Delta\sigma^3)$.
We will show this by analysing a single step of our scheme.
Suppose that~\eqref{eq:ODESys} has a solution $(w, t)$ defined on $\sigma \in [\sigma_0, \sigma_0 + \Delta\sigma]$.
We set $w_0 = w(\sigma_0)$, $w_1 = w(\sigma_0 + \Delta\sigma)$, $t_0 = t(\sigma_0)$, and $t_1 = t(\sigma_0 + \Delta\sigma)$.
Since~\eqref{eq:ODESys} is autonomous, we may assume without loss of generality that $\sigma_0 = 0$.
Before we start, we reiterate that our calculations will be formal in nature.

Start with
\begin{equation}\label{eq:ExactSingleStep}
w_1 = \exp\left(\int_{0}^{\Delta\sigma} A(w(q))\dd{q}\right) w_0.
\end{equation}
Approximate the integral using the left-point method:
\begin{equation*}
w_1 = \exp\left(A(w_0) \Delta\sigma\right) w_0 + \bigO(\Delta\sigma^2).
\end{equation*}
Use the $(0, 1)$-Pad\'e approximation of $\exp$:
\begin{equation*}
w_1 = (I - A(w_0) \Delta\sigma)^{-1} w_0 + \bigO(\Delta\sigma^2).
\end{equation*}
Set $k_1$ to be the unique solution of
\begin{equation*}
(I - A(w_0) \Delta\sigma) k_1 = w_0.
\end{equation*}
We have $w_1 = k_1 + \bigO(\Delta\sigma^2)$.

Before we continue with the higher-order approximation for $w_1$, we will look at computing $t_1$.
We have
\begin{equation*}
t_1 = t_0 + \int_0^{\Delta\sigma} b(w(q)) \dd{q}.
\end{equation*}
We approximate the integral using the trapezoidal method:
\begin{equation*}
t_1 = t_0 + \frac{\Delta\sigma}{2}(b(w_0) + b(w_1)) + \bigO(\Delta\sigma^3).
\end{equation*}
Approximate $w_1$ on the right by $k_1$:
\begin{equation*}
t_1 = t_0 + \frac{\Delta\sigma}{2}(b(w_0) + b(k_1)) + \bigO(\Delta\sigma^3).
\end{equation*}
Now, we move onto using $k_1$ in a higher-order approximation of $w_1$.
Approximate the integral in~\eqref{eq:ExactSingleStep} using the trapezoidal method:
\begin{equation*}
w_1 = \exp\left(\frac{\Delta\sigma}{2} (A(w_0) + A(w_1))\right) w_0 + \bigO(\Delta\sigma^3).
\end{equation*}
Approximate the $w_1$ on the right by $k_1$:
\begin{equation*}
w_1 = \exp\left(\frac{\Delta\sigma}{2} (A(w_0) + A(k_1))\right) w_0 + \bigO(\Delta\sigma^3).
\end{equation*}
Finally, we use the $(0, 2)$-Pad\'e approximation of $\exp$ which is $\exp(X) = (I - X + \frac{1}{2} X^2)^{-1} + \bigO(X^3)$.
However, since matrix multiplication can introduce a large amount of rounding error, we use a partial-fractions decomposition of this approximation as presented in~\cite{KTV1993}:
\begin{equation*}
(I - X + \frac{1}{2} X^2)^{-1} = -2 \Re\left[(X - (i + 1) I)^{-1} i\right],
\end{equation*}
assuming that the entries of $X$ are real.
This leads us to set $k_2$ as the unique solution to
\begin{equation*}
  \left(\frac{\Delta\sigma}{2} (A(w_0) + A(k_1)) - (i + 1) I\right) k_2 = i w_0.
\end{equation*}
Observe that $w_1 = -2 \Re[k_2] + \bigO(\Delta\sigma^3)$.

One could continue this procedure deriving schemes of higher order.
However, we will not pursue this here.

\section{Calculations with the Lemniscate of Bernoulli and a conjecture}

One parametrisation of the Lemniscate of Bernoulli is given by $\beta(\theta) = (x(\theta), y(\theta))$ where
\[
	x(\theta) = \cos \theta / (1+\sin^2 \theta)
\]
and
\[
	y(\theta) = \cos \theta \sin \theta / (1+\sin^2 \theta)
	\,.
\]
Let us start by calculating the arclength derivative.
First,
\begin{align*}
	x_\theta(\theta) &= (-\sin \theta - \sin^3 \theta - 2\sin \theta \cos^2 \theta) / (1+\sin^2 \theta)^2
	\\
	&= (-\sin \theta - \sin^3 \theta - 2\sin \theta + 2\sin^3 \theta) / (1+\sin^2 \theta)^2
	\\
	&= (-3\sin \theta + \sin^3 \theta) / (1+\sin^2 \theta)^2
\end{align*}
and
\begin{align*}
	y_\theta(\theta) &= \cos^2 \theta / (1+\sin^2 \theta) + x_\theta \sin \theta
	\\
	&= (\cos^2 \theta + \cos^2\theta \sin^2\theta - 3\sin^2 \theta + \sin^4\theta ) / (1+\sin^2 \theta)^2
	\\
	&= (\cos^2 \theta + \cos^2\theta \sin^2\theta - 3\sin^2 \theta + \sin^2\theta(1-\cos^2\theta) ) / (1+\sin^2 \theta)^2
	\\
	&= (1 - 3\sin^2 \theta) / (1+\sin^2 \theta)^2
\,.
\end{align*}
The length of $(x_\theta,y_\theta)$, which we call $v$, simplifies to
\begin{align*}
	v^2(\theta) &= x_\theta^2(\theta) + y_\theta^2(\theta)
	\\
	&= (1 - 6\sin^2\theta + 9\sin^4\theta + 9\sin^2\theta - 6\sin^4\theta + \sin^6\theta) / (1+\sin^2 \theta)^4
	\\
	&= (1 + 3\sin^2\theta + 3\sin^4\theta + \sin^6\theta) / (1+\sin^2 \theta)^4
	\\
	&= (1 + \sin^2\theta)^3 / (1+\sin^2 \theta)^4
	\\
	&= 1 / (1+\sin^2 \theta)\,,
\end{align*}
so
\[
	v(\theta) = 1 / (1+\sin^2 \theta)^\frac12\,.
\]
Although we won't use this fact, note that the length of $\beta$ can now be observed to be
\[
	L = \int_0^{2\pi} \frac{1}{\sqrt{1+\sin^2\theta}}\,d\theta
	= 4K(-1)
\]
where $K(-1)$ is the complete elliptic integral of the first kind, with parameter $i$.
This means that if comparing the lifespan of this Lemniscare of Bernoulli with
other curves of winding number zero, they should be scaled to have initial
length equal to $4K(-1)$.

This means that the arclength derivative on $\beta$ is given by
\[
	\partial_s = (1+\sin^2\theta)^\frac12\, \partial_\theta
	\,.
\]
We remark here that this parametrisation has the property that $\beta(\theta) = v^2(\theta)\cos \theta\, (1, \sin \theta)$.
We will use this to make future calculation shorter.

The unit tangent vector is
\begin{align*}
	\partial_s\beta(\theta)
	&= v^{-1}(\theta)\Big(
		2v(\theta)v_\theta(\theta)\cos \theta\, (1, \sin \theta)
		- v^2(\theta)\sin \theta\, (1, \sin \theta)
		+ v^2(\theta)\cos \theta\, (0, \cos \theta)
		\Big)
\\
	&= v^{-1}(\theta)\Big(
		-2v(\theta)
			\frac{\sin\theta \cos\theta}{(1+\sin^2\theta)^\frac32}
			\cos \theta\, (1, \sin \theta)
		+ v^2(\theta)\, (-\sin\theta, \cos^2 \theta - \sin^2 \theta)
		\Big)
\\
	&= v(\theta)\Big(
		 -\sin\theta -2v^2(\theta) \sin\theta \cos^2\theta, \cos^2 \theta - \sin^2 \theta-2v^2(\theta) \sin^2\theta \cos^2\theta
		\Big)
\\
	&= v^3(\theta)\Big(
		 -(1+\sin^2\theta)\sin\theta - 2\sin\theta \cos^2\theta, (1+\sin^2\theta)(\cos^2 \theta - \sin^2 \theta)-2\sin^2\theta \cos^2\theta
		\Big)
\\
	&= v^3(\theta)\Big(
		 -3\sin\theta + \sin^3\theta,
		 1 - \sin^2\theta -\sin^2\theta + \sin^4\theta - \sin^2\theta -\sin^4\theta
		\Big)
\\
	&= v^3(\theta)\Big(
		 -3\sin\theta + \sin^3\theta,
		 1 - 3\sin^2\theta
		\Big)
\,.
\end{align*}
The unit normal vector is the rotation of the unit tangent, with expression
\[
	\nu(\theta) = v^3(\theta) (3\sin^2\theta - 1, \sin^3\theta -3\sin \theta)
	\,.
\]
Now we are interested in establishing the identities that show the lemniscate $\beta$ is a self-similar shrinker under Chen's flow.
Therefore we need to calculate the support function $\IP{\beta}{\nu}$, as well as $k^3$ and $k_{ss}$.
Each of the two latter terms must be proportional to the support function.

We start with:
\begin{align*}
	\IP{\beta(\theta)}{\nu(\theta)}
	&= v^5(\theta){\cos \theta}(3\sin^2\theta - 1 + \sin^4\theta -3\sin^2 \theta)
	\\
	&= v^5(\theta){\cos \theta}(- 1 + \sin^4\theta)
	\\
	&= -v^5(\theta){\cos \theta}(1 - \sin^2\theta)(1 + \sin^2\theta)
	\\
	&= -v^3(\theta){\cos^3 \theta}
\,.
\end{align*}
This is another surprisingly simple expression.

Now we calculate the curvature vector:
\begin{align*}
	\kappa(\theta) &= \tau_s(\theta)
	= v^{-1}(\theta)\, \partial_\theta \bigg[
			v^3(\theta) (-3\sin \theta + \sin^3 \theta, 1 - 3\sin^2 \theta)
		\bigg]
	\\
	&= T(\theta)\tau(\theta)
	+ v^2(\theta) (-3\cos \theta + 3\sin^2 \theta\cos \theta, - 6\sin \theta\cos \theta)
	\\
	&= T(\theta)\tau(\theta)
	+ v^2(\theta) (-3\cos^3 \theta, - 6\sin \theta\cos \theta)
\,.
\end{align*}
We leave the tangential part unsimplified as we will not need it; now, the curvature scalar:
\begin{align*}
	k(\theta) = \IP{\kappa(\theta)}{\nu(\theta)}
	&= v^5(\theta)(-3\cos^3 \theta (3\sin^2-1)  - 6\sin^2 \theta\cos \theta(\sin^2\theta - 3))
	\\
	&= 3v^5(\theta)\cos \theta(\sin^2\theta (3\sin^2-1) - (3\sin^2-1) - 2\sin^2 \theta(\sin^2\theta - 3))
	\\
	&= 3v^5(\theta)\cos \theta(3\sin^4\theta - \sin^2\theta - 3\sin^2\theta + 1 - 2\sin^4\theta + 6\sin^2\theta)
	\\
	&= 3v^5(\theta)\cos \theta(\sin^2\theta + 1 )^2
	\\
	&= 3v(\theta)\cos \theta
\,.
\end{align*}
Therefore
\[
	k^3(\theta)
	= 27v^3(\theta)\cos^3 \theta
	= -27\IP{\beta(\theta)}{\nu(\theta)}
	\,,
\]
and we have shown the first of the identities we need (see \eqref{LB2} in the introduction).

Next, we differentiate the curvature scalar:
\begin{align*}
	k_s
	&= 3v^{-1}(\theta)\partial_\theta\bigg[
		v(\theta)\cos \theta
		\bigg]
	\\
	&= -3\sin \theta
	- 3v^2(\theta)\sin \theta\cos^2 \theta\,,
\end{align*}
and once more,
\begin{align*}
	k_{ss}(\theta)
	&= v^{-1}(\theta)\partial_\theta\bigg[
	 -3\sin \theta
	- 3v^{2}(\theta)\sin\theta\cos^2 \theta
		\bigg]
	\\
	&=
	 -3v^{-1}(\theta)\cos \theta
	+ 6v^{3}(\theta)
		\sin^2 \theta\cos^3 \theta
	- 3v(\theta)
		(\cos^3 \theta - 2\cos \theta\sin^2\theta)
	\\
	&=
	 -3\cos \theta v^{3}(\theta)\Big(
	 	(1 + 2\sin^2\theta + \sin^4\theta)
		+ (- 2\sin^2\theta + 2\sin^4\theta)
		+ (1-\sin^4\theta)
		+ (-2 \sin^2\theta - 2\sin^4\theta)
	 	\Big)
	\\
	&=
	 -3 v^{3}(\theta)\cos \theta
		\Big(
	 	2
		- 2\sin^2\theta
	 	\Big)
	\\
	&=
	 -6v^{3}(\theta)\cos^3 \theta 
	\\
	&= 6\IP{\beta(\theta)}{\nu(\theta)}
\,.
\end{align*}
That's it.
To see that $\beta$ leads to a self-similar Chen flow, set $\gamma:\S\times[0,T)\rightarrow\R^2$ to be the one-parameter family of scalings of $\beta$ with scaling factor $h:[0,T)\rightarrow\R$:
\[
	\gamma(t) = h(t)\beta(\theta)
\,.
\]
Now, supposing that $\gamma$ should solve Chen's flow, we must have
\begin{align*}
	\IP{\gamma_t - (-k_{ss}+k^3)\nu}{\nu}
	&= h'(t)\IP{\beta}{\nu} - h^{-3}(t) (-6\IP{\beta}{\nu} - 27\IP{\beta}{\nu})
	\\
	&= h^{-3}(t)\IP{\beta}{\nu}\Big[
		h^3(t)h'(t) + 33
		\Big]
	\\
	&= \frac{h^{-3}(t)}{4}\IP{\beta}{\nu}\Big[
		(h^4(t))' + 132
		\Big]
	= 0\,.
\end{align*}
In order to ensure this, we may choose
\[
	h(t) = (1-132t)^\frac14
	\,.
\]
Therefore, the map
\[
	(\theta, t) \mapsto 
	h(t) = (1-132t)^\frac14\frac{\cos\theta}{1+\sin^2\theta} \Big(1, \sin\theta \Big)
\]
is a self-similar shrinking solution to Chen's flow, with initial data the Lemniscate of Bernoulli.

Note that this also tells us the lifespan of the solution: $T = \frac1{132}$.
We can therefore make the following conjecture:

\begin{conj}
Suppose $\gamma:\S\times[0,T)\rightarrow\R^2$ is a Chen flow with $\omega=0$ and $L(0) = 4K(-1)$.
Then
\[
	T \le \frac1{132}
\]
and $T = \frac1{132}$ if and only if $\gamma$ is a self-similar shrinking Lemniscate of Bernoulli.
\end{conj}

\bibliographystyle{plain}
\bibliography{chen}
\end{document}